\documentclass[11pt]{amsart}
\usepackage{amsmath, amsfonts, amsthm, amssymb}
\usepackage[all,cmtip,line]{xy}
\usepackage{array}
\usepackage{enumerate}
\usepackage{srcltx} 

\newtheorem{thm}{Theorem}[section]
\newtheorem{lem}{Lemma}[section]
\newtheorem{prop}[lem]{Proposition}
\newtheorem{cor}[lem]{Corollary}

\newtheorem{rem}[lem]{Remark}
\newtheorem{conj}{Conjecture}[section]

\numberwithin{equation}{section}

\newcommand{\bR}{ \mathbb{R}} 

\newcommand \eps{\varepsilon}

\newlength{\originalbase}
\newcommand{\spacing}[1]{\setlength{\baselineskip}{#1\originalbase}}

\begin{document}               

\newcommand{\avint}{{- \hspace{-3.5mm} \int}}

\spacing{1}

\title{ On the constant scalar curvature K\"ahler metrics(I)---\\Apriori estimates }
\author{Xiuxiong Chen, Jingrui Cheng}

\date{\today}
\maketitle

\centerline{Dedicated to Sir Simon Donaldson for his 60th Birthday}

\medskip

\begin{abstract}
In this paper, we derive apriori estimates for constant scalar curvature K\"ahler metrics on a compact K\"ahler manifold. We show that higher order derivatives can be estimated in terms of a $C^0$ bound for the K\"ahler potential.
We also discuss some local versions of these estimates which can be of independent interest.
\end{abstract}


\section{Introduction}
This is the first of a series of three papers in the study of  of  constant scalar curvature K\"ahler metrics (cscK metrics), following a program outlined in \cite{chen15}.
In this paper, we focus on establishing a priori estimates for cscK metrics in compact K\"ahler manifold without boundary. Our estimates can be easily adapted to extremal K\"ahler metrics and for simplicity of presentations, we leave such an extension to the interested readers except to note that for extremal K\"ahler metrics, its scalar curvature is a priori bounded
depending on K\"ahler class.  In the subsequent two papers, we will use these estimates (and its generalizations) to  study the Calabi-Donaldson  theory on the geometry of extremal K\"ahler metrics,
in particular, to establish the celebrated conjecture of Donaldson on geodesic stability as we as the well known
properness conjecture relating the existence of cscK metrics with the properness of K energy functional. \\

 In  \cite{chen15}, the first named author advocates a new continuity path which links the cscK equation to certain second order elliptic equation, apparently inspired by  the success of the classical continuity path for K\"ahler Einstein (KE) metrics and Donaldson's continuity path for conical KE metrics. In general,   apriori estimates are usually the prelude to the success of any continuity path aiming to obtain existence results of cscK metrics since openness
is already established in \cite{chen15}.   \\

Let us recall a conjecture made earlier by the first named author (c.f. \cite{chenhe12}).
\begin{conj}  Let $(M,[\omega]) $ be any compact K\"ahler manifold without boundary. Suppose $ \omega_\varphi$ is a constant scalar curvature K\"ahler metric. If $\varphi$ is uniformly bounded, then any higher derivative estimate of $\varphi$ is also uniformly bounded.
\end{conj}
It is worthwhile to give a brief review of the history of this subject and hopefully,  this will make it self-evident why this conjecture is interesting. 
A special case of constant scalar curvature K\"ahler metric is the well known KE metric which has been the main focus of K\"ahler geometry since the inception of the celebrated Calabi conjecture
on K\"ahler Einstein metrics in 1950s. In 1958, E. Calabi published the fundamental $C^3$ estimate for Monge-Amp$\grave{\text{e}}$re equation \cite{calabi58} which later played a crucial role in Yau's seminal resolution of Calabi conjecture \cite{Yau78}
in 1976 when the first Chern class is either negative or zero (In negative case, T. Aubin has an independent proof) .  This work of Yau is so influential that generations of experts in K\"ahler geometry afterwards largely followed the same route: Securing a $C^0$ estimate first, then move on to obtain $C^2,\; C^3$ estimates etc. 
In the case of positive first Chern class, G. Tian proved Calabi conjecture in 1989 \cite{tian87} for Fano surfaces when the automorphism group is reductive.  It is well known that there are obstructions to the existence of KE metrics in Fano manifolds; around 1980s, Yau proposed a conjecture
which relates the existence of K\"ahler Einstein metrics to the stability of underlying tangent bundles. This conjecture was settled in 2012 through a series of work  
 CDS \cite{cds12-1}  \cite{cds12-2}  \cite{cds12-3} and we refer interested readers to this set of papers for further references in the subject of KE metrics. The proof of CDS's work is itself quite involved as it sits at the intersection of  several different subjects: algebraic geometry, several complex variables, geometry  analysis and metric differential geometry etc.
 
 
 To move beyond CDS's work on K\"ahler Einstein metrics, one direction is the study of the existence problem of cscK metrics which satisfy a 4th order PDE.  The following is a conjecture  which  is a refinement of Calabi's original idea that every K\"ahler class
 must have its own best, canonical representatives.
 
 \begin{conj} [Yau-Tian-Donaldson] Let $[\omega]=C_1({L})$ for some holomorphic line bundle $L$ on a K\"ahler manifold $M$, then the underlying  $(M, L)$ is K-stable if and only if there exists a constant scalar curvature K\"ahler metric in $[\omega]$. 
\end{conj}

One conspicuous and memorable feature of CDS's proof is the heavy use of Cheeger-Colding
theory on manifold with Ricci curvature bounded from below. The apriori bound on Ricci curvature
for KE metrics make such an application of Cheeger Colding theory seamlessly smooth and effective.
 However, if we want to attack this general conjecture,  there will be a  dauntingly high wall to climb
 since there is no a priori bound on Ricci curvatue. Therefore,  the entire Cheeger-Colding theory needs to be re-developed if it is at all feasible. On the other hand, there is a second, less visible but perhaps even more significant feature of CDS's proof is: The whole proof is designed for constant scalar curvature K\"ahler metrics and the use of algebraic criteria and Cheeger Colding theory is to conclude  that the
 a $C^0$ bound holds for K\"ahler potential so that we can apply the apriori estimates for complex KE metrics developed by Calabi, Yau and others.  Indeed, this is exactly how we make use of Cheeger Colding theory and stability condition in CDS's proof
 to nail down a $C^0$  estimate on potential.  Unfortunately, such an estimate is missing in this generality for a 4th order fully nonlinear equation. Indeed, as noted by other famous authors in the subject as well,  the difficulty permeates the cscK theory are two folds:  one cannot use maximal principle from PDE point of view and one can not have much control of metric from the bound of the scalar curvature. \\
 
  In this paper,  we want to tackle this challenge and we prove
\begin{thm}\label{t1.1} If $(M, \omega_\varphi)$ is a cscK metric, where $\omega_{\varphi}=\omega_0+\sqrt{-1}\partial\bar{\partial}\varphi$, 
then all higher derivatives of the K\"ahler potential $\varphi$ can be estimated
 in terms of an upper bound of $\int_M\log\big(\frac{\omega_{\varphi}^n}{\omega_0^n}\big)\omega_{\varphi}^n$.
\end{thm}

As a consequence, we show that
\begin{cor}\label{c1.1}
Let $(M,\omega_{\varphi})$ be as in above theorem, then all higher derivatives of $\varphi$ can be estimated in terms of $||\varphi||_0$.
\end{cor}

The cscK metric equation can be re-written as a pair of coupled equations 

\begin{align}\label{csck1}
&\log\det(g_{\alpha\bar{\beta}}+
\varphi_{\alpha\bar{\beta}})=F+\log\det(g_{\alpha\bar{\beta}}),\\
\label{csck2}
&\Delta_{\varphi}F=-\underline{R}+tr_{\varphi}Ric_g.
\end{align}

Here $\Delta_{\varphi}$ denotes the Laplace operator defined by the  K\"ahler form  $\omega_{\varphi}: $\[
{\mathcal{H}}  = \{\varphi \in C^\infty(M): \;\; \omega_{\varphi}:=\sqrt{-1}(g_{\alpha\bar{\beta}}+\varphi_{\alpha\bar{\beta}})dz_{\alpha}\wedge d\bar{z}_{\beta}>0\}.
\]

The following proposition might be well known to experts (c.f. \cite{chen15}).
\begin{prop}\label{p1.1} If $\frac{1}{C}\omega_0\leq \omega_{\varphi}\leq C\omega_0$, for some constant $C>0$, then all higher derivatives can be estimated in terms of $C$.
\end{prop}

Following \cite{chen15}, Proposition 2.1, we outline some key arguments for this proposition:  since $g_\varphi$ is quasi-isometric, then Equation (1.2) is
uniformly elliptic with a bounded right hand side. Therefore, by De Giorgi-Nash-Moser theory(\cite{GT}, Theorem 8.22), $[F]_{C^{\alpha}(M,g)}$
is uniformly bounded for some $0<\alpha<1$. Substituting this into Equation (1.1),  it becomes a complex Monge-Amp$\grave{\text{e}}$re equation
with $C^\alpha$ bound on the right hand side. Following theory of  Caffarelli, Evans-Krylov(see \cite{YW} for details on extension to complex setting), we know $[\varphi]_{C^{2,\alpha'}(M,g)}$ is uniformly
bounded, for each $\alpha'<\alpha$.
This means (\ref{csck2}) is uniformly elliptic with coefficients in $C^{\alpha'}$. Hence we may apply Schauder theory(\cite{GT}, Theorem 6.2) to conclude an estimate for $||F||_{2,\alpha'}$. 
Now we can go back to (\ref{csck1}). Differentiating the equation, we can conlude $\varphi$ is bounded in $C^{4,\alpha'}$. Hence we may bootstrap this way and get estimates for all higher derivatives. \\

In this short argument, it is obvious that the crucial assumption is that the metric in question is quasi-isometric. The hard challenge is to prove a priori that
the metric $\omega_\varphi$ is quasic isometric. However, there is not much room for improvement at least locally,  following the well known example of Pogorelov on real Monge-Amp$\grave{\text{e}}$re equation. In \cite{He12}, W.Y. He  adapted the construction of Pogorelov's example to complex setting and obtained a complete solution to
\[
\det u_{i\bar j} = 1
\]  
in $\mathbb{C}^n$ which is not $C^2.\;$  Thus, for this conjecture to be true,  the global nature of compact K\"ahler manifold must come into play in a crucial way. \\


Theorem 1.1 can be expanded into a more detailed version.
The constants $C$ in the theorem below can change from line to line. More generally, throughout this paper, the ``C" without subscript may change from line to line, while if there is subscript, then it is some fixed constant.
\begin{thm}\label{t1.2} Suppose $(M, \omega_\varphi)$ is a constant scalar curvature K\"ahler metric. Then the following statements are mutually equivalent: 
\begin{enumerate}
\item There is a constant such that $\int_M \; \log {\omega_\varphi^n\over \omega^n} \cdot \omega_\varphi^n  < C;$
\item There is a constant such that $|\varphi| < C;$
\item  There is a constant $C$ such that  $|\nabla \varphi| < C$ and $\log {\omega_\varphi^n\over \omega^n} \geq -C$; 
\item   There is a constant $C$ such that  $ \frac{1}{C}<{\omega_\varphi^n\over \omega^n} < C;$

\item  There is a constant $C$ such that $n +\Delta \varphi < C$ and $\frac{\omega_{\varphi}^n}{\omega_0^n}>\frac{1}{C}$;
\item All higher derivates of $\varphi$ is uniformly bounded.
\end{enumerate}
\end{thm} 
Some remarks are in order:
\begin{enumerate}
\item The strength of statement is roughly in increasing order. The equivalence of (1) and (6) gives Theorem 1.1.
\item From (5) to (6), this is exactly Proposition 1.1, since this assumption implies $\frac{1}{C'}\omega_0\leq\omega_{\varphi}\leq C'\omega_0$. All other estimates are new.
\item Here is the flow line of our proof: \[
(1) \stackrel{section\, 5}{\Longrightarrow} (2) +(4)\;\; {\rm and} \;\;(3)\stackrel{section\,2}{\Longleftrightarrow}(4)\stackrel{section\,3}{\Longrightarrow}(5)\stackrel{section\,4}{\Longrightarrow}(6).\]
\end{enumerate}

\begin{rem}
In Theorem \ref{t1.1} and the first part of Theorem \ref{t1.2}, it is sufficient to assume that $\varphi$ remains bounded under $L^1$ geodesic distance, due to the fact that cscK metrics are minimizers of $K$-energy.  
We will discuss this matter in more detail in our next paper of the series.
\end{rem}
Now we present technical theorems which lead to this main theorem. Indeed, these
technical theorems are interesting in its own right and may be used in other applications. 
\begin{thm}\label{t1.3new}
(Corollary \ref{c5.3}) Let $\varphi$ be a smooth solution to (\ref{csck1}), (\ref{csck2}), then for any $1<p<\infty$, there exists a constant $C$, depending only on the background K\"ahler metric $(M,g)$, an upper bound of $\int_Me^FFdvol_g$, and $p$, such that
\begin{equation}
||e^F||_{L^p(dvol_g)}\leq C,\,\,||\varphi||_0\leq C.
\end{equation}
\end{thm}

\begin{thm}\label{t1.4new}
(Corollary \ref{c5.4})Let $\varphi$ be a smooth solution to (\ref{csck1}), (\ref{csck2}), then there exists a constant $C$, 
depending only on the background metric $(M,g)$ and an upper bound for $\int_Me^FFdvol_g$, 
such that
\begin{equation}
e^F\leq C.
\end{equation}
\end{thm}

\begin{prop}(Proposition \ref{p2.1})\label{p1.3new}
Let $\varphi$ be a smooth solution to (\ref{csck1}), (\ref{csck2}), then there exists a constant $C$, depending only on $||\varphi||_0$, such that
\begin{equation}
F\geq-C.
\end{equation}
\end{prop}

\begin{thm}\label{t1.5new}(Theorem \ref{t2.2}) Let $\varphi$ be a smooth solution to (\ref{csck1}), (\ref{csck2}), then there exists a constant $C$, depending only on $||\varphi||_0$ and the background metric $g$, such that
\begin{equation}
\frac{|\nabla\varphi|^2}{e^F}\leq C.
\end{equation}
\end{thm}

We also show that one can estimate the upper bound of $F$ directly in terms of gradient bound of $\varphi$.
This result is not directly needed for our main result, but can be of independent interest.
\begin{thm}(Theorem \ref{t2.1})\label{t1.6new}
Let $\varphi$ be a smooth solution to (\ref{csck1}) and (\ref{csck2}), then there exists a constant $C$, depending only on $||\varphi||_0$,  and the backgroud metric $g$, such that
\begin{equation}
\max_Me^{\frac{F}{n}}\leq C\max_M|\nabla\varphi|^2.
\end{equation}
\end{thm}

For second order estimate, Chen-He\cite{chenhe12} establish an a priori bound on $n+\Delta \varphi$ in terms of $|\nabla F|_{L^p} (p > 2n)$ via integral estimate, in absense of (\ref{csck2}).  Inspired by this paper \cite{chenhe12} and utilizing the additional equation (1.2), we are able to obtain a $W^{2,p}$ estimate for any $p>0$, using only $||F||_0$. Theorem \ref{t1.5new} is used essentially in this estimate.

\begin{thm} (Theorem \ref{t4.1}, Corollary \ref{c3.2})\label{t1.7new} Let $\varphi$ be a smooth solution to (\ref{csck1}), (\ref{csck2}), then for any $1<p<\infty$, there exists a constant $\alpha(p)>0$, depending only on $p$, and another constant $C$, depending only on $||\varphi||_0$, the background metric $g$, and $p$,
 such that 
\begin{equation}
\int_Me^{-\alpha(p)F}(n+\Delta\varphi)^p\leq C.
\end{equation}
In particular, $||n+\Delta\varphi||_{L^p(dvol_g)}\leq C'$, where $C'$ has the same dependence as $C$ in this theorem, but additionally on $||F||_0$.
\end{thm}

If we can prove an upper bound for $F$, then the following theorem becomes very interesting.
\begin{thm}(Proposition \ref{p4.2})\label{t1.8new}
Let $\varphi$ be a smooth solution to (\ref{csck1}), (\ref{csck2}).
Then there exists $p_n>1$, depending only on $n$, and a constant $C$, depending on $||\varphi||_0$, $||F||_0$, $||n+\Delta\varphi||_{L^{p_n}(dvol_g)}$, and the background metric $g$, such that
\begin{equation}
n+\Delta\varphi\leq C.
\end{equation}
\end{thm}
It is interesting to compare this result with second derivative estimates for complex Monge-Amp$\grave{e}$re equations. In \cite{LLZ}, the authors obtained $C^{2,\alpha}$ estimates for complex Monge-Amp$\grave{e}$re equation, depending on $C^{1,\beta}$ bound of the solution (with $\beta$ close enough to 1) and $C^{\alpha}$ bound of the right hand side.
In \cite{He12}, the authors obtained $W^{3,p}$ bound of solution to complex Monge-Amp$\grave{e}$re depending only on $C^0$ bound of the solution and $W^{1,p}$ bound of right hand side for $p>2n$.
In this result, we are not assuming any regularity of the right hand side $e^F$, but assumes quite strong bound ($W^{2,p}$ for $p$ large) as a price to pay, and the second equation (\ref{csck-2}) needs to be used in an essential way.

Theorem \ref{t1.8new} is reminiscent to a renowned problem in $\mathbb{C}^n$ which goes back to S. T. Yau, E. Calabi:  whether global solution of Calabi Yau metric in $\mathbb{C}^n$ must be Euclidean metric or not? 
 This problem is disapproved by a nontrivial construction of Calabi Yau metric in $\mathbb{C}^2$ by C. LeBrun. Perhaps one need to strengthen the assumption by assuming it is asymptotically Euclidean at $\infty.\;$ This is made known to be true by G. Tian in dimension $n =2\:$ and conjectured to be true in all dimensions.  While we prepare this paper, it is now known through a surprising result of Y. Li in dimension 3 \cite{LiYang-17} and then
Conlon-Rochon \cite{Conlon-Rochon17}, G. Szekelyhidi \cite{GS2017} in all dimensions that this fails in general. This exciting new development makes statement like Corollary \ref{c1.3new} below more interesting.  This corollary offers a different point of view: If we control asymptotical growth of the underlying metrics, then the rigidity result still hold for scalar flat K\"ahler metrics (in particular Calabi Yau metrics) in $\mathbb{C}^n.\;$

\begin{cor}\label{c1.3new} Let $u$ be a global smooth pluri-subharmonic function such that $\sqrt{-1}\partial\bar{\partial}u$ defines a scalar flat metric on $\mathbb{C}^n$. If for some $p>3n(n-1)$, we have
\[
\displaystyle \liminf_{r\rightarrow \infty} \; {1\over r^{2n}} \displaystyle \int_{B_r(0)\subset \mathbb{C}^n} \; (\Delta u)^p +\big(\sum_k\frac{1}{u_{k\bar{k}}}\big)^p< \infty,
\]
then the Levi Hessian of $u$ is constant. 
\end{cor}
We will prove this corollary in section 6, using a similar argument as Proposition \ref{p6.1}  We observe that this theorem covers the well-known Calabi Yau metric equation $$\det u_{i\bar{j}}=1$$ as a special case. One interesting question is, what is the smallest number $p$ for which this corollary still holds?  In section 6, we also show that when $n=2$, for a solution $\varphi$ of cscK in a domain of $\mathbb{C}^n$, if $|\nabla\varphi|$ is locally bounded, then the volume ratio $ {\omega_\varphi^n\over \omega^n}$ is also bounded from above locally.
It is not clear to us if this estimate can be generalized to higher dimensions.\\

Finally we would like to explain the organization of this paper:

\sloppy  In section 2, we prove Proposition \ref{p1.3new}, Theorem \ref{t1.5new} and Theorem \ref{t1.6new}.

In section 3, we prove Theorem \ref{t1.7new} by iteration, which is a crucial step towards the main result.

In section 4, we use iteration again to improve $L^p$ bound of $n+\Delta\varphi$ for $p<\infty$ to an $L^{\infty}$ bound of $n+\Delta\varphi$, proving Theorem \ref{t1.8new}. 
This estimate requires a bound for $||n+\Delta\varphi||_{L^p}$ for some $p>p_n$ \footnote{  Here $p_n\leq (3n-3)(4n+1)$. But this most likely is not sharp.} , depending on $||F||_0$. The key ingredient is a calculation for $\Delta_{\varphi}(|\nabla_{\varphi}F|^2)$.
Combining the results in section 2, 3, 4 as well as Proposition \ref{p2.1} gives estimate for all higher derivatives in terms of $||\varphi||_0$ and $||F||_0$. 

In Section 2-4, we always assume $|\varphi|$ is a priori bounded.  This assumption is removed in Section 5 where we prove Theorem \ref{t1.3new} and Theorem \ref{t1.4new}.
From these two results, we get estimate for $||F||_0$ and $||\varphi||_0$ depending only on entropy bound of $F$.
The key ingredient is the use of $\alpha$-invariant and the construction of a new test function.
On the other hand, if we start with a bound for $||\varphi||_0$, and use the convexity of $K$-energy along $C^{1,1}$ geodesics, it is relatively easy to get an entropy bound of $F$, hence all higher estimates.

In section 6, we obtain some interior estimates for cscK in a bounded domain of $\mathbb{C}^n$. 
Such estimates are not directly needed for our main results but may be of independent interest. \\

{\bf Acknowledgment} In the Fall of 1997, Sir Simon invited the first named author to join him in exploring the space of K\"ahler potentials, and 
Sir Simon has been remarkably generous about  sharing his time and ideas ever since. It is therefore  a deep pleasure to 
dedicate this paper to Sir Simon Donaldson, in celebration of his 60th birthday, and in acknowledgment of 
the far-reaching influence of his profound mathematical ideas, which have changed the landscape of mathematics so much.

\section{The volume ratio ${\omega_\varphi^n\over \omega^n}$ and $C^1$ bound on K\"ahler potential}
The main theorem in this section is to prove that the first derivative of $\varphi$ is pointwisely controlled by volume ratio $e^F$ from above,
assuming a bound for $||\varphi||_0$. Conversely, the bound for $||\nabla \varphi||_0$ can in turn control $e^F.\;$ However, this control is much
weaker since it is of global nature.  \\

First we show that a $C^0$ bound for $\varphi$ implies a lower bound for $F$.
\begin{prop}\label{p2.1}
Let $(\varphi,F)$ be smooth solutions to cscK, then there exists a positive constant $C_2$ depending only on $||\varphi||_0$ and the upper bound of the Ricci form of the background metric $g$, such that $F\geq -C_2$.
\end{prop}
This proposition first appeared in \cite{HeZeng17}. However, for the convenience of the reader, we include a proof here.
\begin{proof}
This step is relatively easy. Let $p\in M$, we may choose a local normal coordinate in a neighborhood of $p$, such that 
\begin{equation} 
g_{i\bar{j}}(p)=\delta_{ij},\;\;\nabla g_{i\bar{j}}(p)=0,\;\;{\rm  and} \; \;\varphi_{i\bar{j}}(p)=\varphi_{i\bar{i}}(p)\delta_{ij}.\label{eq:normalcoord} \end{equation}
In this paper,  we will always work under this coordinate unless specified otherwise.
Choose the constant $C_{2.1}$ to be $C_{2.1}=2\displaystyle \max_M R_{i\bar{i}}+\frac{2|\underline{R}|}{n}+1$.
Under this coordinate, we can calculate:
\begin{equation}\label{calc}
\begin{split}
\Delta_{\varphi}(F+C_{2.1}\varphi)&=-\underline{R}+\frac{R_{i\bar{i}}}{1+\varphi_{i\bar{i}}}+C_{2.1}\frac{\varphi_{i\bar{i}}}{1+\varphi_{i\bar{i}}}\\
&=-\underline{R}+C_{2.1}n-\frac{C_{2.1}-R_{i\bar{i}}}{1+\varphi_{i\bar{i}}}\leq-\underline{R}+C_{2.1}n-\frac{nC_{2.1}}{2}e^{-\frac{F}{n}}\\
&\leq2C_{2.1}n-\frac{nC_{2.1}}{2}e^{-\frac{F}{n}}.
\end{split}
\end{equation}
In the second line above, we used the arithemetic-geometric inequality:
$$
\frac{1}{n}\sum_i\frac{1}{1+\varphi_{i\bar{i}}}\geq\Pi_i(1+\varphi_{i\bar{i}})^{-\frac{1}{n}}=e^{-\frac{F}{n}}.
$$
Now let $p_0$ be such that the function $F+C_{2.1}\varphi$ achieves minimum at $p_0$, then from (\ref{calc}), we see
\begin{equation}
0\leq 2C_{2.1}n-\frac{nC_{2.1}}{2}e^{-\frac{F}{n}}(p_0).
\end{equation}
This gives a lower bound for $F$, depending only on $C^0$ bound for $\varphi$.

\end{proof}
Next we move on to estimate the upper bound of $F$ in terms of $\displaystyle \max_M |\nabla\varphi|$.
\begin{thm}\label{t2.1} Let $(\varphi,F)$ be smooth solutions to cscK, then there exists a constant $C_{2.2}$, depending only on $||\nabla\varphi||_0$ and lower bound of bisectional cruvature of the background metric $g$, such that $F\leq C_{2.2}$.
\end{thm}
\begin{proof}
The argument uses maximum principle again. This time we will calculate $\Delta_{\varphi}(e^{F-\lambda\varphi}(K+|\nabla\varphi|^2))$ where $\lambda,K>0$ are constants to be determined later. We choose a normal coordinate (equation (\ref{eq:normalcoord}))  and do the following calculations.
We have

\begin{eqnarray}\label{start}
\qquad\Delta_{\varphi} (e^{F-\lambda\varphi}(K+|\nabla\varphi|^2)) &= & \Delta_{\varphi}(e^{F-\lambda\varphi})(K+|\nabla\varphi|^2) +e^{F-\lambda\varphi}\Delta_{\varphi}(K+|\nabla\varphi|^2)\\
& &+e^{F-\lambda\varphi}\cdot\frac{(F_i-\lambda\varphi_i)(|\nabla\varphi|^2)_{\bar{i}}+(F_{\bar{i}}-\lambda\varphi_{\bar{i}})(|\nabla\varphi|^2)_{i}}{1+\varphi_{i\bar{i}}}. \nonumber
\end{eqnarray}

We can first calculate:
\begin{equation}
\begin{split}
\Delta_{\varphi}(e^{F-\lambda\varphi})&=e^{F-C\varphi}\frac{|F_i-\lambda\varphi_i|^2}{1+\varphi_{i\bar{i}}}+e^{F-\lambda\varphi}(\Delta_{\varphi}F-\frac{\lambda\varphi_{i\bar{i}}}{1+\varphi_{i\bar{i}}})\\
&=e^{F-\lambda\varphi}\frac{|F_i-\lambda\varphi_i|^2}{1+\varphi_{i\bar{i}}}+ e^{F-\lambda\varphi}(-\underline{R}-\lambda n+\frac{\lambda+R_{i\bar{i}}}{1+\varphi_{i\bar{i}}}).
\end{split}
\label{eq:volumeratio}
\end{equation}
By differentiating equation (\ref{csck1}) in $z_{\alpha}$ direction, we obtain 
\begin{equation}  \displaystyle \sum_{i}\frac{\varphi_{i\bar{i}\alpha}}{1+\varphi_{i\bar{i}}}=F_{\alpha} \;\;{\rm and}\;\;  \displaystyle \sum_{i}\frac{\varphi_{i\bar{i} \bar\alpha}}{1+\varphi_{i\bar{i}}}=F_{\bar \alpha}.  \label{eq: gradF}\end{equation}
  Then we calculate
\begin{equation}
\begin{split}
\Delta_{\varphi}(|\nabla\varphi|^2)&=\frac{R_{\alpha\bar{\beta}i\bar{i}}\varphi_{\alpha}\varphi_{\bar{\beta}}}{1+\varphi_{i\bar{i}}}+\frac{|\varphi_{i\alpha}|^2}{1+\varphi_{i\bar{i}}}+\frac{\varphi^2_{i\bar{i}}}{1+\varphi_{i\bar{i}}}+
\frac{\varphi_{\alpha}\varphi_{\bar{\alpha}i\bar{i}}+\varphi_{\bar{\alpha}}\varphi_{\alpha i\bar{i}}}{1+\varphi_{i\bar{i}}}\\
&\geq-\frac{C_{2.21}|\nabla\varphi|^2}{1+\varphi_{i\bar{i}}}+\frac{|\varphi_{i\alpha}|^2}{1+\varphi_{i\bar{i}}}+\frac{\varphi^2_{i\bar{i}}}{1+\varphi_{i\bar{i}}}+\varphi_{\alpha}F_{\bar{\alpha}}+\varphi_{\bar{\alpha}}F_{\alpha}\\
&=-\frac{C_{2.21}|\nabla\varphi|^2}{1+\varphi_{i\bar{i}}}+\frac{|\varphi_{i\alpha}|^2}{1+\varphi_{i\bar{i}}}+\frac{\varphi^2_{i\bar{i}}}{1+\varphi_{i\bar{i}}}+
\varphi_{\alpha}(F_{\bar{\alpha}}-\lambda\varphi_{\bar{\alpha}})+
\varphi_{\bar{\alpha}}(F_{\alpha}-\lambda\varphi_{\alpha})+2\lambda|\varphi_{\alpha}|^2\\
&\geq-\frac{C_{2.21}|\nabla\varphi|^2}{1+\varphi_{i\bar{i}}}+\frac{|\varphi_{i\alpha}|^2}{1+\varphi_{i\bar{i}}}+\frac{\varphi^2_{i\bar{i}}}{1+\varphi_{i\bar{i}}}-\eps(n+\Delta\varphi)-\frac{|F_{\alpha}-\lambda\varphi_{\alpha}|^2|\varphi_{\alpha}|^2}{\eps(1+\varphi_{\alpha\bar{\alpha}})}+2\lambda|\varphi_{\alpha}|^2.
\end{split}
\label{eq:firstderivative}
\end{equation}
Here $C_{2.21}$ depends only on lower bound of bisectional curvature of $g$.

For the last term in (\ref{start}), we estimate in the following way:
\begin{eqnarray}\label{crossterm1}
&& \frac{|(F_{\bar{i}}-\lambda\varphi_{\bar{i}})(|\nabla\varphi|^2)_i|}{1+\varphi_{i\bar{i}}} \\& = & \frac{|(F_{\bar{i}}-\lambda\varphi_{\bar{i}})(\varphi_{\alpha}\varphi_{\bar{\alpha}i}+\varphi_{\bar{\alpha}}\varphi_{\alpha i})|}{1+\varphi_{i\bar{i}}} \nonumber \\
&\leq & \frac{|(F_{\bar{i}}-\lambda\varphi_{\bar{i}})\varphi_i\varphi_{i\bar{i}}|}{1+\varphi_{i\bar{i}}}+\frac{|(F_{\bar{i}}-\lambda\varphi_{\bar{i}})\varphi_{\bar{\alpha}}\varphi_{\alpha i}|}{1+\varphi_{i\bar{i}}} \nonumber\\
&\leq &\frac{|F_i-\lambda\varphi_i|^2|\varphi_i|^2}{2\eps(1+\varphi_{i\bar{i}})}+\frac{|F_i-\lambda\varphi_i|^2|\varphi_{\alpha}|^2}{2\eps(1+\varphi_{i\bar{i}})}+\frac{\eps\varphi_{i\bar{i}}^2}{2(1+\varphi_{i\bar{i}})}+\frac{\eps|\varphi_{i\alpha}|^2}{2(1+\varphi_{i\bar{i}})}. \nonumber
\end{eqnarray}

The other conjugate term satisfies the same estimate as above. Combining above calculations, we obtain:

\begin{eqnarray}\label{60}
 && \frac{\Delta_{\varphi}(e^{F-\lambda\varphi}(K+|\nabla\varphi|^2))}{e^{F-\lambda\varphi}}\\
& \geq & \big(K+|\nabla\varphi|^2-3\eps^{-1}|\nabla\varphi|^2\big)\frac{|F_i-\lambda\varphi_i|^2}{1+\varphi_{i\bar{i}}} \nonumber \\
& & \qquad \qquad +\frac{(\lambda+R_{i\bar{i}})(K+|\nabla\varphi|^2)-C_{2.21}|\nabla\varphi|^2}{1+\varphi_{i\bar{i}}}+\frac{(1-\eps)\varphi_{i\bar{i}}^2}{1+\varphi_{i\bar{i}}}-\eps(n+\Delta\varphi) \nonumber \\
& & \qquad \qquad  \qquad \qquad+\frac{|\varphi_{i\alpha}|^2(1-\eps)}{1+\varphi_{i\bar{i}}}+(-\underline{R}-\lambda n)(K+|\nabla\varphi|^2). \nonumber
\end{eqnarray}
Now it's time to choose the constants $\eps$, $\lambda$, and $K$ appearing above.

First we choose $\eps=\frac{1}{4}$.
With this choice, we have 
\begin{equation}
\begin{split}
\sum_i\frac{(1-\eps)\varphi_{i\bar{i}}^2}{1+\varphi_{i\bar{i}}}-\eps(n+\Delta\varphi)&=\frac{3}{4}(n+\Delta\varphi)-\frac{3n}{2}+\frac{3}{4}\sum_i\frac{1}{1+\varphi_{i\bar{i}}}-\frac{1}{4}(1+\varphi_{i\bar{i}})\\
&\geq\frac{1}{2}(n+\Delta\varphi)-\frac{3n}{2}.
\end{split}
\end{equation}
Then we choose $\lambda$ so large that $\lambda+R_{i\bar{i}}>1$.
Finally, we choose $K$ so large that 
\begin{align}\label{K1}
&K>C_{2.21}\max_M|\nabla\varphi|^2\\
\label{K2}
&K>3\eps^{-1}\max_{M}|\nabla\varphi|^2=12\max_M|\nabla\varphi|^2.
\end{align}
With above choices  for $\lambda$ and $K$, we have
\begin{equation}
(\lambda+R_{i\bar{i}})(K+|\nabla\varphi|^2)-C_{2.21}|\nabla\varphi|^2\geq K-C_{2.21}|\nabla\varphi|^2>0.
\end{equation}
and also
\begin{equation}
K-3\eps^{-1}|\nabla\varphi|^2>0.
\end{equation}
Hence we conclude from (\ref{60}) that
\begin{equation}\label{nn}
\Delta_{\varphi}(e^{F-\lambda\varphi}(K+|\nabla\varphi|^2))\geq e^{F-\lambda\varphi}\big(-(|\underline{R}|+\lambda n)(K+\max_M|\nabla\varphi|^2)-C_{2.22}+\frac{1}{4}(n+\Delta\varphi)\big).
\end{equation}

Denote $v=e^{F-C\varphi}(K+|\nabla\varphi|^2)$, it is enough to show $v$ has an upper bound. we see from (\ref{nn}) that there exists constants $C_{2.23}>0$, $C_{2.24}>0$, possibly depending on $\displaystyle \max_M|\nabla\varphi|^2$, such that 
\begin{equation}\label{71}
\Delta_{\varphi}(v)\geq v(-C_{2.23}+\frac{1}{C_{2.24}}(n+\Delta\varphi)).
\end{equation}

Here we notice that $n+\Delta\varphi\geq n e^{\frac{F}{n}}$. Hence we obtain from (\ref{71}) that 
\begin{equation}
\Delta_{\varphi}(v)\geq v(-C_{2.23}+\frac{1}{C_{2.24}}e^{\frac{F}{n}}).
\end{equation}
Let the maximum of $v$ be achieved at point $p$, then we know $-C_{2.23}+\frac{e^{\frac{F}{n}(p)}}{C_{2.24}}\leq0$.
This gives an upper bound of $F$ at $p_0$, hence an upper bound for $v$, where this bound depends on $\displaystyle \max_M|\nabla\varphi|$.
\end{proof}

Conversely, we have the following key estimate, which will be needed when we do the $W^{2,p}$ estimates of $\varphi$.

\begin{thm}\label{t2.2} There exists a constant $C_{2.3}$, depending only on $||\varphi||_0$, lower bound of bisectional curvature and upper bound of Ricci form of $g$, such that $${{|\nabla\varphi|^2}\over e^{F}} \leq C_{2.3}.$$
\end{thm}
\begin{proof}
We will consider $\Delta_{\varphi}(e^{-(F+\lambda\varphi)+\frac{1}{2}\varphi^2}(|\nabla\varphi|^2+K))$.
Here $\lambda>0$, $K>0$ are constants to be determined below.
Then we have

\begin{eqnarray}\label{2.25}
&& \Delta_{\varphi}(e^{-(F+\lambda\varphi)+\frac{1}{2}\varphi^2}(|\nabla\varphi|^2+K))\\& = & \Delta_{\varphi}(e^{-(F+\lambda\varphi)+\frac{1}{2}\varphi^2})(|\nabla\varphi|^2+K) +e^{-(F+\lambda\varphi)+\frac{1}{2}\varphi^2}\Delta_{\varphi}(|\nabla\varphi|^2)\nonumber \\
& &\qquad \qquad+\frac{2e^{-(F+\lambda\varphi)+\frac{1}{2}\varphi^2}}{1+\varphi_{i\bar{i}}}Re\big((-F_i-\lambda\varphi_i+\varphi\varphi_i)(|\nabla\varphi|^2)_{\bar{i}}\big). \nonumber
\end{eqnarray}
For simplicity of notation, set $$A(F,\varphi)=-(F+\lambda\varphi)+\frac{1}{2}\varphi^2.$$
Similar as before, we may calculate:

\begin{eqnarray}\label{2.17}
&&\Delta_{\varphi}(e^{A(F,\varphi)}) \\&= & e^A\frac{|-F_i-\lambda\varphi_i+\varphi\varphi_i|^2}{1+\varphi_{i\bar{i}}}+e^A\big(-\Delta_{\varphi}(F+\lambda\varphi)+\varphi\Delta_{\varphi}\varphi\big)+e^A\frac{|\varphi_i|^2}{1+\varphi_{i\bar{i}}}\nonumber \\
&= & e^A\frac{|-F_i-\lambda\varphi_i+\varphi\varphi_i|^2}{1+\varphi_{i\bar{i}}}+e^A\bigg(\underline{R}-\lambda n+n\varphi+\sum_i\frac{\lambda-R_{i\bar{i}}-\varphi}{1+\varphi_{i\bar{i}}}\bigg)+\frac{e^A|\varphi_i|^2}{1+\varphi_{i\bar{i}}}.\nonumber
\end{eqnarray}
Recall the calculation in (\ref{eq:firstderivative}):
\begin{equation}\label{2.27}
\begin{split}
\Delta_{\varphi}(|\nabla\varphi|^2)&=\frac{R_{i\bar{i}\alpha\bar{\beta}}\varphi_{\alpha}
\varphi_{\bar{\beta}}}{1+\varphi_{i\bar{i}}}+\frac{|\varphi_{i\alpha}|^2}{1+\varphi_{i\bar{i}}}+\frac{\varphi_{i\bar{i}}^2}{1+\varphi_{i\bar{i}}}+
F_{\bar{\alpha}}\varphi_{\alpha}+
F_{\alpha}\varphi_{\bar{\alpha}}\\
&\geq-C_{2.21}|\nabla\varphi|^2\sum_i\frac{1}{1+\varphi_{i\bar{i}}}+\frac{|\varphi_{i\alpha}|^2}{1+\varphi_{i\bar{i}}}+\frac{\varphi_{i\bar{i}}^2}{1+\varphi_{i\bar{i}}}
\\
&+(-2\lambda+2\varphi)|\nabla\varphi|^2+2Re\big((F_{\alpha}+\lambda\varphi_{\alpha}-\varphi\varphi_{\alpha})\varphi_{\bar{\alpha}}\big)
.
\end{split}
\end{equation}
Again $C_{2.21}$ depends only on curvature bound of $g$.
Also
$$
(|\nabla\varphi|^2)_{\bar{i}}=\varphi_{\alpha}\varphi_{\bar{\alpha}\bar{i}}+
\varphi_{\bar{i}}\varphi_{i\bar{i}},\,\,(|\nabla\varphi|^2)_i=\varphi_{\bar{\alpha}}\varphi_{\alpha i}+\varphi_i\varphi_{i\bar{i}}.
$$
Hence if we plug in (\ref{2.17}) and (\ref{2.27}) back to (\ref{2.25}), we obtain:

\begin{equation}\label{2.19}
\begin{split}
&\quad \Delta_{\varphi}(e^A(|\nabla\varphi|^2+K))e^{-A}\\ & \geq   |\nabla_{\varphi}(F+\lambda\varphi)-\varphi\nabla_{\varphi}\varphi|^2(|\nabla\varphi|^2+K)+|\nabla_{\varphi}\varphi|^2(|\nabla\varphi|^2+K)\\
&+\big(\underline{R}-\lambda n+n\varphi+\sum_i\frac{\lambda-R_{i\bar{i}}-\varphi}{1+\varphi_{i\bar{i}}}\big)(|\nabla\varphi|^2+K)+\frac{-C_{2.21}|\nabla\varphi|^2+|\varphi_{i\alpha}|^2+\varphi_{i\bar{i}}^2}{1+\varphi_{i\bar{i}}}\\
& +(-2\lambda +2\varphi)|\nabla\varphi|^2+2Re\big((F_{\alpha}+\lambda\varphi_{\alpha}-\varphi\varphi_{\alpha})\varphi_{\bar{\alpha}}\big)\\
&\qquad\qquad\qquad\qquad\qquad+\frac{2Re\big((-F_i-\lambda\varphi_i+\varphi\varphi_i)(\varphi_{\alpha}\varphi_{\bar{\alpha}\bar{i}}+\varphi_{\bar{i}}\varphi_{i\bar{i}})\big)}{1+\varphi_{i\bar{i}}}. 
\end{split}
\end{equation}
We notice the following complete square in the above sum:
\begin{equation}\label{2.20}
\begin{split}
&\frac{1}{1+\varphi_{i\bar{i}}}|\varphi_{i\alpha}-\big(F_i+\lambda\varphi_i-\varphi\varphi_i\big)\varphi_{\alpha}|^2 \\
&=\frac{|\varphi_{i\alpha}|^2}{1+\varphi_{i\bar{i}}} +\frac{2Re\big((-F_i-\lambda\varphi_i+\varphi\varphi_i)\varphi_{\alpha}\varphi_{\bar{\alpha}\bar{i}}\big)}{1+\varphi_{i\bar{i}}}+\frac{|-F_i-\lambda\varphi_i+\varphi\varphi_i|^2|\nabla\varphi|^2}{1+\varphi_{i\bar{i}}}.
\end{split}
\end{equation}
We will drop this complete square in the following.
Next we observe a crucial cancellation, which is the key point of this argument. We look at the last two terms in (\ref{2.19}) and observe:
\begin{equation}\label{2.22}
(F_{\alpha}+\lambda\varphi_{\alpha}-\varphi\varphi_{\alpha})\varphi_{\bar{\alpha}}+\frac{(-F_i-\lambda\varphi_i+\varphi\varphi_i)\varphi_{\bar{i}}\varphi_{i\bar{i}}}{1+\varphi_{i\bar{i}}}=\frac{(F_i+\lambda\varphi_i-\varphi\varphi_i)\varphi_{\bar{i}}}{1+\varphi_{i\bar{i}}}.
\end{equation}
Hence we have
\begin{equation}\label{2.23new}
\begin{split}
&\Delta_{\varphi}(e^A(|\nabla\varphi|^2+K))e^{-A}\geq K\frac{|-F_i-\lambda\varphi_i+\varphi\varphi_i|^2}{1+\varphi_{i\bar{i}}}+\frac{|\varphi_i|^2(|\nabla\varphi|^2+K)}{1+\varphi_{i\bar{i}}}\\
& +\sum_i\frac{\lambda-R_{i\bar{i}}-\varphi}{1+\varphi_{i\bar{i}}}(|\nabla\varphi|^2+K)+\bigg(\underline{R}-\lambda n+n\varphi\bigg)(|\nabla\varphi|^2+K)\\
&  -C_{2.21}|\nabla\varphi|^2\sum_i\frac{1}{1+\varphi_{i\bar{i}}}+\frac{\varphi_{i\bar{i}}^2}{1+\varphi_{i\bar{i}}}+(-2\lambda+2\varphi)|\nabla\varphi|^2+2Re\bigg(\frac{(F_i+\lambda\varphi_i-\varphi\varphi_i)\varphi_{\bar{i}}}{1+\varphi_{i\bar{i}}}\bigg).
\end{split}
\end{equation}

Now we make the choices of $\lambda$, $K$.
We choose $\lambda=10(\sup_M|R_{i\bar{i}}|+||\varphi||_0+C_{2.21}+1)$ and $K=10$.
With this choice, we now estimate the terms in (\ref{2.23new}), with various constants $C_i$ which depends only on the curvature bound of $g$ and $||\varphi||_0$.
\begin{equation}
\big(\underline{R}-\lambda n+ n \varphi\big)(|\nabla\varphi|^2+K)\geq -C_{2.31}(|\nabla\varphi|^2+1).
\end{equation}
\begin{equation}
(-2\lambda+2\varphi)|\nabla\varphi|^2\geq -C_{2.32}|\nabla\varphi|^2.
\end{equation}
\begin{equation}
\begin{split}
\frac{|(F_i+\lambda\varphi_i-\varphi\varphi_i)\varphi_{\bar{i}}|}{1+\varphi_{i\bar{i}}}&\leq\frac{1}{2}\frac{|F_i+\lambda\varphi_i-\varphi\varphi_i|^2}{1+\varphi_{i\bar{i}}}+\frac{1}{2}\frac{|\varphi_i|^2}{1+\varphi_{i\bar{i}}}\\
&\leq\frac{1}{2}\frac{|F_i+\lambda\varphi_i-\varphi\varphi_i|^2}{1+\varphi_{i\bar{i}}}+\frac{1}{2}|\nabla\varphi|^2\sum_i\frac{1}{1+\varphi_{i\bar{i}}}.
\end{split}
\end{equation}
\begin{equation}
\sum_i\frac{\lambda-R_{i\bar{i}}-\varphi}{1+\varphi_{i\bar{i}}}(|\nabla\varphi|^2+K)-C_{2.21}|\nabla\varphi|^2\sum_i\frac{1}{1+\varphi_{i\bar{i}}}\geq10|\nabla\varphi|^2\sum_i\frac{1}{1+\varphi_{i\bar{i}}}.
\end{equation}
\begin{equation}
\frac{\varphi_{i\bar{i}}^2}{1+\varphi_{i\bar{i}}}\geq0.
\end{equation}
Combining all these estimates, we obtain from (\ref{2.23new}) that
\begin{equation}\label{2.35}
\Delta_{\varphi}(e^A(|\nabla\varphi|^2+K))\geq e^A\big(\frac{|\varphi_i|^2|\nabla\varphi|^2}{1+\varphi_{i\bar{i}}}+9|\nabla\varphi|^2\sum_i\frac{1}{1+\varphi_{i\bar{i}}}-C_{2.33}(|\nabla\varphi|^2+1)\big).
\end{equation}
Here $C_{2.33}$ depends only on curvature bound of $g$ and $||\varphi||_0$.  Using Young's inequality, we have,
\begin{equation*}
\begin{split}
|\nabla\varphi|^{\frac{2}{n}}e^{-\frac{F}{n}}&\leq\sum_i\frac{|\varphi_i|^{\frac{2}{n}}}{(1+\varphi_{i\bar{i}})^{\frac{1}{n}}}\cdot(1+\varphi_{i\bar{i}})^{\frac{1}{n}}e^{-\frac{F}{n}}\\
&\leq\frac{1}{n}\sum_i\frac{|\varphi_i|^2}{1+\varphi_{i\bar{i}}}+\frac{n-1}{n}\sum_i(1+\varphi_{i\bar{i}})^{\frac{1}{n-1}}e^{-\frac{F}{n-1}}\\
&\leq(n-1)\big(\frac{|\varphi_i|^2}{1+\varphi_{i\bar{i}}}+\frac{1}{n}\sum_i(1+\varphi_{i\bar{i}})^{\frac{1}{n-1}}e^{-\frac{F}{n-1}}\big)\\
&\leq (n-1)\big(\frac{|\varphi_i|^2}{1+\varphi_{i\bar{i}}}+(n+\Delta\varphi)^{\frac{1}{n-1}}e^{-\frac{F}{n-1}}\big)\\
&\leq (n-1)\big(\frac{|\varphi_i|^2}{1+\varphi_{i\bar{i}}}+\sum_i\frac{1}{1+\varphi_{i\bar{i}}}\big).
\end{split}
\end{equation*}
Thus,
\begin{equation*}
\begin{split}
\frac{|\varphi_i|^2|\nabla\varphi|^2}{1+\varphi_{i\bar{i}}}+|\nabla\varphi|^2\sum_i\frac{1}{1+\varphi_{i\bar{i}}}
&\geq\frac{1}{n-1}|\nabla\varphi|^{2+\frac{2}{n}}e^{-\frac{F}{n}}.
\end{split}
\end{equation*}

Hence we get from (\ref{2.35}) that
\begin{equation}
\begin{split}
\Delta_{\varphi}(e^A(|\nabla\varphi|^2+K))&\geq e^{-C\varphi+\frac{1}{2}\varphi^2}\big(e^{-(1+\frac{1}{n})F}|\nabla\varphi|^{2+\frac{2}{n}}-C_{2.33}e^{-F}|\nabla\varphi|^2-C_{2.33}e^{-F}\big)\\
&=e^{-C\varphi+\frac{1}{2}\varphi^2}\big((e^{-F}|\nabla\varphi|^2)^{1+\frac{1}{n}}-C_{2.33}e^{-F}|\nabla\varphi|^2-C_{2.33}e^{-F}\big).
\end{split}
\end{equation}

Suppose that the function $e^{-(F+C\varphi)+\frac{1}{2}\varphi^2}(|\nabla\varphi|^2+K)$ achieves maximum at $p$.
Then at point $p$, we have
\begin{equation}\label{2.37}
0\geq(e^{-F}|\nabla\varphi|^2)^{1+\frac{1}{n}}-C_{2.33}e^{-F}|\nabla\varphi|^2-C_{2.33}e^{-F}.
\end{equation}
Recall Proposition \ref{p2.1} gives an estimate for $e^{-F}$ which depends only on $||\varphi||_0$ and the curvature bound of $g$. Therefore, we get a bound for $e^{-F}|\nabla\varphi|^2(p)$ with the same dependence.
Hence we have a bound for $e^{-(F+\lambda\varphi)+\frac{1}{2}\varphi^2}(|\nabla\varphi|^2+K)(p)$, with the dependence as stated in the theorem.

But this function achieves maximum at $p$, so we are done.
\end{proof}

\section{The volume ratio $\frac{\omega_{\varphi}^n}{\omega_0^n}$ and $W^{2,p}$ bound on K\"ahler potential }

In this section, we prove

\begin{thm}\label{t4.1} For any $p>0$, there exist constants $\alpha(p)>0$, $C(p)>0$, so that
\begin{equation}\label{W2p}
\int_Me^{-\alpha(p)F}(n+\Delta\varphi)^pdvol_g\leq C(p).
\end{equation}
Here $\alpha(p)$ depends only on $p$(can be explicitly calculated). The constant $C_p$ depends only on $p$, $||\varphi||_0$, the upper bound of Ricci form, lower bound of the bisectional curvature, and volume of $g$.
\end{thm}
\begin{proof}
One start by calculating:
\begin{equation}\label{2.88}
\begin{split}
\Delta_{\varphi}(e^{-\alpha(F+\lambda\varphi)}(n+&\Delta\varphi))=\Delta_{\varphi}(e^{-\alpha(F+\lambda\varphi)})(n+\Delta\varphi)+e^{-\alpha(F+\lambda\varphi)}\Delta_{\varphi}(n+\Delta\varphi)\\
&+e^{-\alpha(F+\lambda\varphi)}(-\alpha)\frac{(F_i+\lambda\varphi_i)(\Delta\varphi)_{\bar{i}}+(F_{\bar{i}}+\lambda\varphi_{\bar{i}})(\Delta\varphi)_i}{1+\varphi_{i\bar{i}}}.
\end{split}
\end{equation}
If we choose $\lambda>2\sup Ric$, then
\begin{equation}
\begin{split}
\Delta_{\varphi}&(e^{-\alpha(F+\lambda\varphi)})=\frac{\alpha^2|F_i+\lambda\varphi_i|^2}{1+\varphi_{i\bar{i}}}e^{-\alpha(F+\lambda\varphi)}+\alpha e^{-\alpha(F+\lambda\varphi)}(\underline{R}-\lambda n+\sum_i\frac{\lambda-R_{i\bar{i}}}{1+\varphi_{i\bar{i}}})\\
&\geq \frac{\alpha^2|F_i+\lambda\varphi_i|^2}{1+\varphi_{i\bar{i}}}e^{-\alpha(F+\lambda\varphi)}
+\alpha e^{-\alpha(F+\lambda\varphi)}(\underline{R}-\lambda n)+\frac{\lambda\alpha}{2}e^{-\alpha(F+\lambda\varphi)}\sum_i\frac{1}{1+\varphi_{i\bar{i}}}.
\end{split}
\end{equation}
For the term $\Delta_{\varphi}(n+\Delta\varphi)$, we choose a normal coordinate (c.f. equation (\ref{eq:normalcoord})) and then follow Yau's calculation\cite{Yau78}. First, note that
\begin{equation}
\Delta_{\varphi}(n+\Delta\varphi)=\frac{1}{1+\varphi_{k\bar{k}}}
\bigg(g^{i\bar{j}}\varphi_{i\bar{j}}\bigg)_{k\bar{k}}=\frac{R_{i\bar{i}k\bar{k}}\varphi_{i\bar{i}}}{1+\varphi_{k\bar{k}}}+\frac{\varphi_{k\bar{k}i\bar{i}}}{1+\varphi_{k\bar{k}}}.
\end{equation}
We wish to represent the $4$-th derivative of $\varphi$ in terms of $F$. For this we take equation (\ref{csck1}) and differentiate it twice in $z_i$, $z_{\bar{i}}$ and then sum over $i =1,2 \cdots n.\;$  We obtain:
\begin{equation}\label{1.32new}
\frac{\varphi_{k\bar{k}i\bar{i}}}{1+\varphi_{k\bar{k}}}-\frac{R_{k\bar{k}i\bar{i}}}{1+\varphi_{k\bar{k}}}-\frac{|\varphi_{k\bar{\beta}i}|^2}{(1+\varphi_{k\bar{k}})(1+\varphi_{\beta\bar{\beta}})}=F_{i\bar{i}}-R_{i\bar{i}}.
\end{equation}
Hence
\begin{equation}\label{3.6new}
\begin{split}
\Delta_{\varphi}(n&+\Delta\varphi)=\frac{R_{k\bar{k}i\bar{i}}(1+\varphi_{k\bar{k}})}{1+\varphi_{i\bar{i}}}+\frac{|\varphi_{p\bar{q}i}|^2}{(1+\varphi_{p\bar{p}})(1+\varphi_{q\bar{q}})}+\Delta F-R\\
&\geq-C_{3.1}(n+\Delta\varphi)\sum_i\frac{1}{1+\varphi_{i\bar{i}}}+\frac{|\varphi_{p\bar{q}i}|^2}{(1+\varphi_{p\bar{p}})(1+\varphi_{q\bar{q}})}+\Delta F-R.
\end{split}
\end{equation}
Here $C_{3.1}$ depends only on curvature bound of $g$ and $R$ is the scalar curvature of the background metric $g$. Plug in to equation (\ref{2.88}) and  we get
\begin{equation}\label{3.7}
\begin{split}
\Delta_{\varphi}&(e^{-\alpha(F+\lambda\varphi)}(n+\Delta\varphi))\geq e^{-\alpha(F+\lambda\varphi)}(\frac{\lambda\alpha}{2}-C_{3.1})(n+\Delta\varphi)\sum_i\frac{1}{1+\varphi_{i\bar{i}}}\\
&+\alpha e^{-\alpha(F+\lambda\varphi)}(\underline{R}-\lambda n)(n+\Delta\varphi)+e^{-\alpha(F+\lambda\varphi)}(\Delta F-R).
\end{split}
\end{equation}
Here we already drop the term:
\begin{equation*}
\begin{split}
& \frac{\alpha^2|F_i+\lambda\varphi_i|^2}{1+\varphi_{i\bar{i}}} (n+\Delta \varphi)   + (-\alpha)\frac{(F_i+\lambda\varphi_i)(\Delta\varphi)_{\bar{i}}+(F_{\bar{i}}+\lambda\varphi_{\bar{i}})(\Delta\varphi)_i}{1+\varphi_{i\bar{i}}} + \frac{|\varphi_{p\bar{q}i}|^2}{(1+\varphi_{p\bar{p}})(1+\varphi_{q\bar{q}})}\\
 \quad\geq  &\frac{\alpha^2|F_i+\lambda\varphi_i|^2}{1+\varphi_{i\bar{i}}}(n+\Delta\varphi)-2\alpha Re\bigg(\frac{(F_i+\lambda\varphi_i)(\Delta\varphi)_{\bar{i}}}{1+\varphi_{i\bar{i}}}\bigg)+\frac{|(\Delta\varphi)_i|^2}{(n+\Delta\varphi)(1+\varphi_{i\bar{i}})}\\
= & \frac{n+\Delta\varphi}{1+\varphi_{i\bar{i}}}|\alpha(F_i+\lambda\varphi_i)-\frac{(\Delta\varphi)_i}{n+\Delta\varphi}|^2\geq0.
 \end{split}
\end{equation*}
From the first line to second line in the above, we observed that
$$
\frac{|(\Delta\varphi)_i|^2}{1+\varphi_{i\bar{i}}}  = \frac{|\sum_p\varphi_{p\bar{p}i}|^2}{1+\varphi_{i\bar{i}}}
 \leq  \frac{|\varphi_{p\bar{p}i}|^2(n+\Delta\varphi)}{(1+\varphi_{i\bar{i}})(1+\varphi_{p\bar{p}})}\\  \leq  \frac{|\varphi_{p\bar{q}i}|^2(n+\Delta\varphi)}{(1+\varphi_{p\bar{p}})(1+\varphi_{q\bar{q}})}.
 $$
Set $$ u=e^{-\alpha(F+\lambda\varphi)}(n+\Delta\varphi) $$  
and note that 
$$
(n+\Delta\varphi)\sum_i\frac{1}{1+\varphi_{i\bar{i}}}\geq e^{-\frac{F}{n-1}}(n+\Delta\varphi)^{1+\frac{1}{n-1}},
$$
 then we know from equation (\ref{3.7}):
\begin{equation} \label{3.8}
\begin{split}
\Delta_{\varphi}&u\geq e^{-(\alpha+\frac{1}{n-1})F-\alpha \lambda\varphi}(\frac{\lambda\alpha}{2}-C_{3.1})(n+\Delta\varphi)^{1+\frac{1}{n-1}}-\alpha e^{-\alpha(F+\lambda\varphi)}(\lambda n-\underline{R})(n+\Delta\varphi)\\
&\qquad \qquad +e^{-\alpha(F+\lambda\varphi)}(\Delta F-R).
\end{split}
\end{equation}
We use the following equality, which holds for any $p\geq0$:
$$
\begin{array}{lcl}
{1\over {2p+1}} \Delta_{\varphi}(u^{2p+1}) & = & 2pu^{2p-1}|\nabla_{\varphi}u|_{\varphi}^2+ u^{2p}\Delta_{\varphi}u \\
& =& 2p u^{2p-2}e^{-\alpha(F+\lambda\varphi)}(n+\Delta\varphi)|\nabla_{\varphi}u|_{\varphi}^2 + u^{2p}\Delta_{\varphi}u \\
& \geq & 2p u^{2p-2}|\nabla u|^2e^{-\alpha(F+\lambda\varphi)} + u^{2p}\Delta_{\varphi}u.
\end{array}
$$
Integrate with respect to $dvol_{\varphi}=e^Fdvol_g$ and plug inequality (\ref{3.8}) to get:
\begin{equation}\label{2.93}
\begin{split}
&\int_M2pu^{2p-2}|\nabla u|^2e^{(1-\alpha)F-\alpha \lambda\varphi}  dvol_g\\
&\qquad\qquad\qquad\qquad+\int_Me^{-(\alpha-\frac{n-2}{n-1})F-\alpha \lambda\varphi}(\frac{\lambda\alpha}{2}-C_{3.1})(n+\Delta\varphi)^{1+\frac{1}{n-1}}u^{2p}dvol_g\\
&\qquad \qquad +\int_Me^{(1-\alpha)F-\alpha \lambda\varphi}u^{2p}\Delta Fdvol_g\leq \int_M\alpha e^{(1-\alpha)F-\alpha \lambda\varphi}(\lambda n-\underline{R})(n+\Delta\varphi)u^{2p}dvol_g\\
&\qquad \qquad \qquad \qquad \qquad \qquad +\int_Me^{(1-\alpha)F-\lambda\alpha\varphi}Ru^{2p}dvol_g.
\end{split}
\end{equation}
We need to handle the term involving $\Delta F$, which is done by integrating by parts.
\begin{equation}\label{2.94}
\begin{split}
&\int_Me^{(1-\alpha)F-\alpha \lambda\varphi}u^{2p}\Delta Fdvol_g=\int_M(\alpha-1)e^{(1-\alpha)F-\alpha \lambda\varphi}u^{2p}|\nabla F|^2dvol_g\\
&+\int_M\alpha \lambda e^{(1-\alpha)F-\lambda\alpha\varphi}u^{2p}\nabla\varphi\cdot\nabla Fdvol_g-\int_M2pe^{(1-\alpha)F-\lambda\alpha\varphi}u^{2p-1}\nabla u\cdot\nabla Fdvol_g.
\end{split}
\end{equation}
Also we can estimate the last term of (\ref{2.94})
\begin{equation}\label{2.96}
u^{2p-1}\nabla u\cdot\nabla F\leq \frac{1}{2}u^{2p-2}|\nabla u|^2+\frac{1}{2}u^{2p}|\nabla F|^2.
\end{equation}
Then we estimate the second to last term of (\ref{2.94}) and obtain:
\begin{equation}\label{2.97}
\begin{split}
\alpha &\lambda e^{(1-\alpha)F-\lambda\alpha\varphi}u^{2p}\nabla\varphi\cdot\nabla F\\
\leq & \frac{\alpha-1}{2}e^{(1-\alpha)F-\lambda\alpha\varphi}u^{2p}|\nabla F|^2 +\frac{\alpha^2\lambda^2}{2(\alpha-1)}u^{2p}|\nabla\varphi|^2e^{(1-\alpha)F-\lambda\alpha\varphi}\\
\leq & \frac{\alpha-1}{2}e^{(1-\alpha)F-\lambda\alpha\varphi}u^{2p}|\nabla F|^2+C_{2.3}\frac{\alpha^2\lambda^2}{2(\alpha-1)}u^{2p}e^{(2-\alpha)F-\lambda\alpha\varphi}.
\end{split}
\end{equation}
When estimating $|\nabla\varphi|^2$ above, we used Theorem \ref{t2.2}, and $C_{2.3}$ is the constant given by that theorem. Plug (\ref{2.96}), (\ref{2.97}) back into (\ref{2.94}), we obtain
\begin{equation}\label{2.98}
\begin{split}
&\int_Me^{(1-\alpha)F-\alpha \lambda\varphi}u^{2p}\Delta Fdvol_g\geq\int_M(\frac{\alpha-1}{2}-p)e^{(1-\alpha)F-\alpha \lambda\varphi}u^{2p}|\nabla F|^2dvol_g\\
&-\int_MC_{2.3}\frac{\alpha^2\lambda^2}{2(\alpha-1)}e^{(2-\alpha)F-\lambda\alpha\varphi}u^{2p}dvol_g-\int_Mpe^{(1-\alpha)F-\lambda\alpha\varphi}u^{2p-2}|\nabla u|^2dvol_g.
\end{split}
\end{equation}
Plug (\ref{2.98}) back to (\ref{2.93}), we see
\begin{equation}
\begin{split}
& \int_Mpe^{(1-\alpha)F-\lambda\alpha\varphi}u^{2p-2}|\nabla u|^2dvol_g + \int_M(\frac{\alpha-1}{2}-p)e^{(1-\alpha)F-\alpha \lambda\varphi}u^{2p}|\nabla F|^2dvol_g\\
&\qquad \qquad +\int_Me^{-(\alpha-\frac{n-2}{n-1})F-\alpha \lambda\varphi}(\frac{\lambda\alpha}{2}-C_{3.1})(n+\Delta\varphi)^{1+\frac{1}{n-1}}u^{2p}dvol_g\\
&\leq \int_M\alpha e^{(1-\alpha)F-\alpha \lambda\varphi}(\lambda n-\underline{R})(n+\Delta\varphi)u^{2p}dvol_g+C_{2.3}\frac{\alpha^2\lambda^2}{2(\alpha-1)}\int_Me^{(2-\alpha)F-\lambda\alpha\varphi}u^{2p}dvol_g\\
&\qquad \qquad+\int_Me^{(1-\alpha)F-\lambda\alpha\varphi}Ru^{2p}dvol_g.
\end{split}
\end{equation}
Now let $\alpha>2p+1$ and $\lambda \alpha \geq 2C_{3.1}+1$, note that $n+\Delta \varphi $ has positive lower bound, then we find from above:
\begin{equation}
\begin{split}
&\int_Me^{-(\alpha-\frac{n-2}{n-1})F-\alpha \lambda\varphi}(n+\Delta\varphi)^{1+\frac{1}{n-1}}u^{2p}dvol_g \\
& \qquad \leq C_{3.2}\alpha\int_Me^{(1-\alpha)F-\alpha \lambda\varphi}(n+\Delta\varphi)u^{2p}dvol_g +C_{3.2}\frac{\alpha^2}{\alpha-1}\int_Me^{(2-\alpha)F-\lambda\alpha\varphi}u^{2p}dvol_g.
\end{split}
\end{equation}
Recall the definition of $u$, this means for any $p\geq0$, $\alpha\geq 2p+2$:
\begin{equation}
\begin{split}
&\int_M\exp(-(2p+1)\alpha F+\frac{n-2}{n-1}F-\lambda\alpha(2p+1)\varphi)(n+\Delta\varphi)^{2p+1+\frac{1}{n-1}}dvol_g\\
&\leq C_{3.2}\alpha\int_M\exp(-(2p+1)\alpha F+F-(2p+1)\alpha \lambda\varphi)(n+\Delta\varphi)^{2p+1}dvol_g\\
&+C_{3.2}\frac{\alpha^2}{\alpha-1}\int_M\exp(-(2p+1)\alpha F+2F-(2p+1)\lambda\varphi)(n+\Delta\varphi)^{2p}dvol_g.
\end{split}
\end{equation}
Hence for some constant $C_{3.3}$ which depends on $||\varphi||_0$, $\alpha$, and $p$, we get:
\begin{equation}\label{2.102}
\begin{split}
&\int_M\exp(-(2p+1)\alpha F+\frac{n-2}{n-1}F)(n+\Delta\varphi)^{2p+1+\frac{1}{n-1}}dvol_g\\
&\leq C_{3.3}\bigg(\int_M\exp(-(2p+1)\alpha F+F)(n+\Delta\varphi)^{2p+1}dvol_g\\
&+\int_M\exp(-(2p+1)\alpha F+2F)(n+\Delta\varphi)^{2p}dvol_g\bigg).
\end{split}
\end{equation}
Start from $p=0$, and take $\alpha=2$, one obtains from (\ref{2.102}) that:
\begin{equation}
\begin{split}
\int_Me^{-\frac{n}{n-1}F}&(n+\Delta\varphi)^{\frac{n}{n-1}}dvol_g\leq C_{3.3}\big(\int_Me^{-F}(n+\Delta\varphi)dvol_g+\int_M dvol_g\big)\\
&\leq C_{3.3}\big(n||e^{-F}||_0 vol(M)+vol(M)\big).
\end{split}
\end{equation}
 Since we obtained in Proposition \ref{p2.1} a bound for $e^{-F}$ depending only on $||\varphi||_0$ and curvature bound of $g$. Hence we get a bound for $\int_Me^{-\frac{n}{n-1}F}(n+\Delta\varphi)^{\frac{n}{n-1}}dvol_g$.

We now claim that there exists a sequence of pair of positive numbers $(p_k, \gamma_k)$ where $p_k\rightarrow \infty$ such that 
\[
\displaystyle \int_M\; e^{-\gamma_k F} (n+\Delta \varphi)^{2 p_k + 1} \;d vol_g < \infty
\] 
for all $k=1,2\cdots .$  Now we explain how we choose this sequence of pairs of positive numbers successively:  In general, suppose we already choose $(p_k,\gamma_k)$ such that the preceding
inequality holds. Choose $\alpha_{k+1}$ sufficiently large such that $$  \alpha_{k+1}\geq 2p_k+2, \qquad {\rm and}\qquad -(2p_k+1)\alpha_{k+1}+2\leq -\gamma_k. $$  Set $\alpha=\alpha_{k+1}$, $p=p_k$ in  (\ref{2.102}), we obtain 
\begin{equation}
\begin{split}
&\int_M\exp(-(2p_k+1)\alpha_{k+1}F+\frac{n-2}{n-1}F)(n+\Delta\varphi)^{2p_k+1+\frac{1}{n-1}}dvol_g\\
&\leq C_{3.31}\bigg(\int_Me^{-\gamma_k F}(n+\Delta\varphi)^{2p_k+1}dvol_g+\int_Me^{-\gamma_kF}(n+\Delta\varphi)^{2p_k}dvol_g\bigg)\\
&\leq C_{3.32}\int_Me^{-\gamma_kF}(n+\Delta\varphi)^{2p_k+1}dvol_g.
\end{split}
\end{equation}
In the second inequality, we used again the fact that $e^{-F}$ is bounded in terms of $||\varphi||_0$ and $g$.
In the last inequality above, we noticed the fact that $n+\Delta\varphi\geq e^{\frac{F}{n}}$, and $e^F$ is bounded from below.
Set  \[ \gamma_{k+1}=(2p_k+1)\alpha_{k+1}-\frac{n-2}{n-1} \qquad {\rm and}\qquad p_{k+1}=p_k+\frac{1}{2(n-1)}.\]
 Then
 \[
 \int_Me^{-\gamma_{k+1}F}(n+\Delta\varphi)^{2p_{k+1}+1}dvol_g \leq C \int_Me^{-\gamma_kF}(n+\Delta\varphi)^{2p_k+1}.\]
 where the constant depends on  $||\varphi||_0$ and the background metric $g$. Our claim is then verified.

By induction, we then get a bound for $\int_Me^{-\gamma_pF}(n+\Delta\varphi)^pdvol_g$ for any $p>0$ and some constant $\gamma_p>0$. Here $\gamma_p$ grows like $p^2$ as $p\rightarrow\infty$.
\end{proof}
\begin{rem}
By a more careful inspection of above argument, one sees that it is possible to choose $\gamma_p=\max( (p+1)(p+2), n p) $, but this is probably not sharp.
\end{rem}
As an immediate consequence, we have the following $W^{2,p}$ estimate of $\varphi$ in terms of $||F||_0$.
\begin{cor}\label{c3.2}
For any $1<p<\infty$, there exist constants $\tilde{C}(p)>0$, depending on $||\varphi||_0$, $||F||_0$, the background metric $g$(with dependence in a way described in Theorem \ref{t4.1}) and $p$, such that $||n+\Delta\varphi||_{L^p}\leq \tilde{C}(p)$.
\end{cor}

\section{$C^{1,1}$ bound of the K\"ahler potential in terms of its $W^{2,p}$ bound}
In this section, we want to prove 
\begin{thm}\label{t3.1} There exists a constant $C_4$, depending only on $||\varphi||_0$, $||F||_0$, the absolute and first derivative bound of the Ricci form, and also the Soboloev constant of the background metric $g$, such that $n+\Delta\varphi\leq C_4$.
\end{thm}
 In view of Theorem 2.1, we have the following immediate consequence:
\begin{cor}
There exists a constant $C_{4.1}$, depending only on $||\varphi||_0$, $||\nabla\varphi||_0$, the background metric $g$(described as in Theorem \ref{t3.1}), such that $n+\Delta\varphi\leq C_{4.1}$.
\end{cor}

 With this assumption, we know from Corollary \ref{c3.2} that for any $p>0$, there exists constants $C_p$, depending on $||\varphi||_0$, $||F||_0$, and the background metric $g$, such that
\begin{equation}
||n+\Delta\varphi||_{L^p(M)}\leq \tilde{C}(p).
\end{equation}
Hence it suffices to prove the following statement:
\begin{prop}\label{p4.2}
Let $(\varphi,F)$ be a smooth solution to cscK, then there exists $p_n>0$, depending only on $n$, such that 
\begin{equation}
\max_M|\nabla_{\varphi}F|_{\varphi}+\max_M(n+\Delta\varphi)\leq C_{4.2}.
\end{equation}
Here $C_4$ depends only on $||F||_0$, $||n+\Delta\varphi||_{L^{p_n}(M)}$, and metric $g$(in the way described in Theorem \ref{t3.1}).
\end{prop}
\begin{rem}\label{r3.2}
From the argument below, one can explicitly get an upper bound for 
$$p_n\leq (3n-3)(4n+1).$$ This upper bound is probably not sharp.
\end{rem}
\begin{proof}  Let us first  calculate $\Delta_{\varphi}(|\nabla_{\varphi}f |_{\varphi}^2)$ for any smooth function $f$ in $M.\;$ First we do the calculation under an orthonormal frame $g_{\varphi}$.
\[
\begin{array}{lcl} \Delta_\varphi |\nabla_\varphi f|^2 & = & (f_i f_{\bar i})_{,j \bar j}\\
& = & f_{,i j \bar j} f_{\bar i} + f_i f_{,\bar i j \bar j} + |f_{,ij}|_{\varphi}^2 + |f_{, i\bar j}|_{\varphi}^2\\
& = &  f_{,j i \bar j} f_{\bar i} + f_i f_{,j \bar i  \bar j} + |f_{,ij}|_{\varphi}^2 + |f_{, i\bar j}|_{\varphi}^2\\
& = &  (\Delta_\varphi f)_i f_{\bar i} + f_i (\Delta_\varphi f)_{\bar i} + Ric_{\varphi,i\bar j} f_j f_{\bar i} +|f_{,ij}|_{\varphi}^2 + |f_{, i\bar j}|_{\varphi}^2.
\end{array}
\]
In the above, $f_{,ij\cdots}$ denote covariant derivatives under the metric $g_{\varphi}$. 
Let $B(\lambda):\bR\rightarrow\bR$ be a smooth function, now we calculate $\Delta_{\varphi}(e^{B(f)}|\nabla_{\varphi}f|_{\varphi}^2)$.

\begin{equation}\label{4.3nn}
\begin{split}
&e^{-B(f)} \cdot \Delta_\varphi (e^{B(f)} |\nabla_\varphi f|_{\varphi}^2) \\
&=  \Delta_\varphi (|\nabla_\varphi f|_{\varphi}^2) + B' (f_i (|\nabla_\varphi f|_{\varphi}^2)_{\bar i} + f_{\bar i} (|\nabla_\varphi f|_{\varphi}^2)_{ i})\\
&\quad\quad\quad\quad + \left((B'^2 + B'') |\nabla_\varphi f|^2 + B' \Delta_\varphi f\right) |\nabla_\varphi f|_{\varphi}^2\\
& =   (\Delta_\varphi f)_i f_{\bar i} + f_i (\Delta_\varphi f)_{\bar i} + Ric_{\varphi,i\bar j} f_j f_{\bar i} +|f_{,ij}|_{\varphi}^2 + |f_{, i\bar j}|_{\varphi}^2\\
 & + B'\left( f_i f_j f_{,\bar j\bar i} + f_i f_{,j \bar i} f_{\bar j} +f_{\bar i} f_j  f_{,\bar ji} + f_{\bar i} f_{,j  i} f_{\bar j} \right) \\
&\quad\quad\quad\quad\quad\quad+\left((B'^2 + B'') |\nabla_\varphi f|_{\varphi}^2 + B' \Delta_\varphi f\right) |\nabla_\varphi f|_{\varphi}^2\\
&\geq   (\Delta_\varphi f)_i f_{\bar i} + f_i (\Delta_\varphi f)_{\bar i} + Ric_{\varphi,i\bar j} f_j f_{\bar i} + |f_{, i\bar j}|_{\varphi}^2\\
&  + B'\left(  f_i f_{,j \bar i} f_{\bar j} +f_{\bar i} f_j f_{,\bar ji} \right) +\left( B'' |\nabla_\varphi f|_{\varphi}^2 + B' \Delta_\varphi f\right) |\nabla_\varphi f|_{\varphi}^2.
\end{split}
\end{equation}

In the inequality above, we noticed and dropped the following complete square:
\begin{equation*}
B'^2|\nabla_{\varphi}f|_{\varphi}^4+B'f_if_jf_{,\bar{j}\bar{i}}+B'f_{\bar{i}}f_{\bar{j}}f_{,ij}+|f_{,ij}|_{\varphi}^2=|f_{,ij}+B'f_if_j|_{\varphi}^2.
\end{equation*}

 We apply above calculation to $F$. Notice that
 \[
Ric_{\varphi, i \bar j} = R_{i \bar j} - F_{i\bar j}.
\]
Set $B' = {1\over 2},\;$and we switch to normal coordinate of $g$ (c.f. (\ref{eq:normalcoord})), 
then we have

\begin{equation}\label{1.240}
\begin{split}
e^{-{F\over 2}} \Delta_{\varphi}(e^{F\over 2}|&\nabla_{\varphi}F|^2)\geq {{ (\Delta_\varphi F)_i F_{\bar i} + (\Delta_\varphi F)_{\bar i }F_{i}}\over {1 +\varphi_{i \bar i}}} +\frac{R_{j\bar{i}}F_iF_{\bar{j}}}{(1+\varphi_{i\bar{i}})(1+\varphi_{j\bar{j}})} \\
& \qquad \qquad +\frac{|F_{i\bar{\alpha}}|^2}{(1+\varphi_{i\bar{i}})(1+\varphi_{\alpha\bar{\alpha}})}+{1\over 2} \Delta_{\varphi}F |\nabla_\varphi F|^2_{\varphi} .
\end{split}
\end{equation}
 
Next we wish to use the equation satisfied by $F$:
\[
\Delta_\varphi F = - \underline{R} + tr_\varphi Ric.
\]
Take derivative with respect to $z_i$ on both sides, we obtain:
\[\begin{array}{lcl}
(\Delta_\varphi F)_{i} & = &  - g_\varphi^{k\bar p} g_{\varphi, \bar p q i} g_\varphi^{q \bar l} R_{k \bar l} + g_\varphi^{k \bar l} R_{k\bar l, i}\\
& = &  -  {{\varphi_{\bar k l i } R_{k \bar l}}\over {(1+\varphi_{k\bar k}) (1+ \varphi_{l \bar l})}} + {R_{k\bar k, i} \over {1 +\varphi_{k \bar k}}}. \end{array}
\]
Plugging this into equation (\ref{1.240}),  we have
\begin{equation}\label{1.24}
\begin{split}
e^{-{F\over 2}} \Delta_{\varphi}(e^{F\over 2}|&\nabla_{\varphi}F|^2)\geq-\frac{F_{\bar{i}}\varphi_{\beta\bar{\alpha}i}R_{\alpha\bar{\beta}}+F_i\varphi_{\beta\bar{\alpha}\bar{i}}R_{\alpha\bar{\beta}}}{(1+\varphi_{i\bar{i}})(1+\varphi_{\alpha\bar{\alpha}})(1+\varphi_{\beta\bar{\beta}})}+\frac{F_{\bar{i}}R_{\alpha\bar{\alpha},i}+F_iR_{\alpha\bar{\alpha},\bar{i}}}{(1+\varphi_{\alpha\bar{\alpha}})(1+\varphi_{i\bar{i}})}\\
&+\frac{R_{j\bar{i}}F_iF_{\bar{j}}}{(1+\varphi_{i\bar{i}})(1+\varphi_{j\bar{j}})}
+\frac{|F_{i\bar{\alpha}}|^2}{(1+\varphi_{i\bar{i}})(1+\varphi_{\alpha\bar{\alpha}})}+{1\over 2} (-\underline{R}+\frac{R_{i\bar{i}}}{1+\varphi_{i\bar{i}}}) |\nabla_\varphi F|^2_{\varphi} .
\end{split}
\end{equation}
In the above, $\varphi_{\beta\bar{\alpha}i}$, $R_{\alpha\bar{\alpha},i}$ etc are just usual derivatives taken under the coordinate as specified above. 
Notice that there will be no more terms like $\frac{F_{j\bar{i}}F_iF_{\bar{j}}}{(1+\varphi_{i\bar{i}})(1+\varphi_{j\bar{j}})}$, because the choice $B'\equiv\frac{1}{2}$ makes such terms exactly cancel out.
Now we proceed further from (\ref{1.24}). 
As preparation, we observe that for any $1\leq i\leq n$:
\begin{equation}\label{1.25}
\frac{1}{1+\varphi_{i\bar{i}}}=e^{-F}\Pi_{j\neq i}(1+\varphi_{j\bar{j}})\leq e^{-F}(n+\Delta\varphi)^{n-1}.
\end{equation}
First we can estimate as follows, with various constants $C_i$ depending only on $n$, $||F||_0$, and the curvature bound of the original metric $g$.
\begin{equation}\label{1.26}
|{1\over 2} e^{{F\over 2}}(-\underline{R}+\frac{R_{i\bar{i}}}{1+\varphi_{i\bar{i}}})|\leq C_{4.3}(1+\sum_i\frac{1}{1+\varphi_{i\bar{i}}})\leq C_{4.3}(1+ne^{-F}(n+\Delta\varphi)^{n-1}).
\end{equation}
\begin{equation}\label{1.27}
\begin{split}
\frac{|F_{\bar{i}}\varphi_{\beta\bar{\alpha}i}R_{\alpha\bar{\beta}}|}{(1+\varphi_{i\bar{i}})(1+\varphi_{\alpha\bar{\alpha}})(1+\varphi_{\beta\bar{\beta}})}&\leq\frac{1}{2}\frac{|\varphi_{\beta\bar{\alpha}i}|^2|R_{\alpha\bar{\beta}}|^2}{(1+\varphi_{\alpha\bar{\alpha}})(1+\varphi_{\beta\bar{\beta}})}+\frac{1}{2}\frac{|F_{\bar{i}}|^2}{(1+\varphi_{i\bar{i}})^2(1+\varphi_{\alpha\bar{\alpha}})(1+\varphi_{\beta\bar{\beta}})}\\
&\leq C_{4.4}\frac{|\varphi_{\beta\bar{\alpha}i}|^2}{(1+\varphi_{\alpha\bar{\alpha}})(1+\varphi_{\beta\bar{\beta}})}+\frac{C_{4.4}|F_i|^2}{1+\varphi_{i\bar{i}}}(n+\Delta\varphi)^{3n-3}.
\end{split}
\end{equation}
In the second line of above estimate, we used (\ref{1.25}) to estimate the extra powers of $\frac{1}{1+\varphi_{\alpha\bar{\alpha}}}$.
The conjugate term will satisfy the same estimate as above.
\begin{equation}\label{1.28}
\begin{split}
|\frac{F_{\bar{i}}R_{\alpha\bar{\alpha},i}}{(1+\varphi_{\alpha\bar{\alpha}})(1+\varphi_{i\bar{i}})}|&\leq\frac{1}{2}\frac{|F_{\bar{i}}|^2|R_{\alpha\bar{\alpha},i}|^2}{1+\varphi_{i\bar{i}}}+\frac{1}{2}\frac{1}{(1+\varphi_{\alpha\bar{\alpha}})^2(1+\varphi_{i\bar{i}})}\\
&\leq C_{4.5}\frac{|F_i|^2}{1+\varphi_{i\bar{i}}}+C_{4.5}(n+\Delta\varphi)^{3n-3}.
\end{split}
\end{equation}
Finally 
\begin{equation}\label{1.29}
\begin{split}
|\frac{R_{j\bar{i}}F_iF_{\bar{j}}}{(1+\varphi_{i\bar{i}})(1+\varphi_{j\bar{j}})}|&\leq\frac{1}{2}\frac{|R_{j\bar{i}}|^2|F_i|^2}{(1+\varphi_{i\bar{i}})(1+\varphi_{j\bar{j}})}+\frac{1}{2}\frac{|R_{j\bar{i}}|^2|F_{\bar{j}}|^2}{(1+\varphi_{i\bar{i}})(1+\varphi_{j\bar{j}})}\\
&\leq C_{4.6}\frac{|F_i|^2}{1+\varphi_{i\bar{i}}}(n+\Delta\varphi)^{n-1}.
\end{split}
\end{equation}
Now combining the estimates in (\ref{1.26}), (\ref{1.27}), (\ref{1.28}), (\ref{1.29}), we obtain from (\ref{1.24}):
\begin{equation}\label{1.30}
\begin{split}
\Delta_{\varphi}(e^{\frac{1}{2}F}|\nabla_{\varphi}F|_{\varphi}^2)&\geq-C_{4.7}(n+\Delta\varphi)^{3n-3}|\nabla_{\varphi}F|_{\varphi}^2-C_{4.7}\frac{|\varphi_{\beta\bar{\alpha}i}|^2}{(1+\varphi_{\alpha\bar{\alpha}})(1+\varphi_{\beta\bar{\beta}})}\\
&-C_{4.7}(n+\Delta\varphi)^{3n-3}+\frac{1}{C_{4.7}}\frac{|F_{i\bar{\alpha}}|^2}{(1+\varphi_{i\bar{i}})(1+\varphi_{\alpha\bar{\alpha}})}.
\end{split}
\end{equation}
Note that $n+\Delta\varphi \geq n e^{\frac{F}{n}}$, hence has a positive uniform lower bound since $F$ is bounded from below. Here $C_{4.7}$ is some constant depending only on $n$, $||F||_0$, and the curvature bound of the original metric $g$.
In order to handle the second term on the right hand side, we need to consider $\Delta_{\varphi}(n+\Delta\varphi)$.
For this we can recall our calculation in (\ref{3.6new}):
\begin{equation}\label{1.33}
\begin{split}
\Delta_{\varphi}(n+\Delta\varphi)&=\frac{R_{i\bar{i}\alpha\bar{\alpha}}(1+\varphi_{i\bar{i}})}{1+\varphi_{\alpha\bar{\alpha}}}+\frac{|\varphi_{\alpha\bar{\beta}i}|^2}{(1+\varphi_{\alpha\bar{\alpha}})(1+\varphi_{\beta\bar{\beta}})}+F_{i\bar{i}}-R_{i\bar{i}}\\
&\geq\frac{-C_{4.8}(1+\varphi_{i\bar{i}})}{1+\varphi_{\alpha\bar{\alpha}}}+\frac{|\varphi_{\alpha\bar{\beta}i}|^2}{(1+\varphi_{\alpha\bar{\alpha}})(1+\varphi_{\beta\bar{\beta}})}+F_{i\bar{i}}-C_{4.8}\\
&\geq-C_{4.9}(n+\Delta\varphi)^n+\frac{|\varphi_{\alpha\bar{\beta}i}|^2}{(1+\varphi_{\alpha\bar{\alpha}})(1+\varphi_{\beta\bar{\beta}})}+F_{i\bar{i}}-C_{4.8}.
\end{split}
\end{equation}
Let $K>0$ be a constant, we combine (\ref{1.30}), (\ref{1.33}), and conclude:
\begin{equation}
\begin{split}
\Delta_{\varphi}(e^{\frac{1}{2}F}&|\nabla_{\varphi}F|_{\varphi}^2+K(n+\Delta\varphi))\geq -C_{4.7}(n+\Delta\varphi)^{3n-3}|\nabla_{\varphi}F|_{\varphi}^2+(K-C_{4.7})\frac{|\varphi_{\alpha\bar{\beta}i}|^2}{(1+\varphi_{\alpha\bar{\alpha}})(1+\varphi_{\beta\bar{\beta}})}\\
&-C_{4.7}(n+\Delta\varphi)^{3n-3}-KC_{4.9}(n+\Delta\varphi)^n+KF_{i\bar{i}}-K C_{4.8}+\frac{1}{C_{4.7}}\frac{|F_{i\bar{\alpha}}|^2}{(1+\varphi_{i\bar{i}})(1+\varphi_{\alpha\bar{\alpha}})}.
\end{split}
\end{equation}
First we choose $K=C_{4.7}+1$, and calculate:
\begin{equation}
|KF_{i\bar{i}}|\leq\frac{1}{2C_{4.7}}\frac{|F_{i\bar{i}}|^2}{(1+\varphi_{i\bar{i}})^2}+\frac{K^2C_{4.7}}{2}(1+\varphi_{i\bar{i}})^2\leq\frac{1}{2C_{4.7}}\frac{|F_{i\bar{i}}|^2}{(1+\varphi_{i\bar{i}})^2}+\frac{nK^2C_{4.7}}{2}(n+\Delta\varphi)^2.
\end{equation}
Hence there exists a constant $C_{4.91}$, with the same dependence as said above, such that
\begin{equation}
\begin{split}
\Delta_{\varphi}&(e^{\frac{1}{2}F}|\nabla_{\varphi}F|_{\varphi}^2+K(n+\Delta\varphi))\geq-C_{4.91}(n+\Delta\varphi)^{3n-3}|\nabla_{\varphi}F|_{\varphi}^2-C_{4.91}(n+\Delta\varphi)^{3n-3}\\
&\geq-C_{4.91}e^{-\frac{1}{2}F}(n+\Delta\varphi)^{3n-3}(e^{\frac{1}{2}F}|\nabla_{\varphi}F|_{\varphi}^2+K(n+\Delta\varphi))-C_{4.91}(n+\Delta\varphi)^{3n-3}.
\end{split}
\end{equation}
Set
\[ u=e^{\frac{1}{2}F}|\nabla_{\varphi}F|_{\varphi}^2+K(n+\Delta\varphi),\]
 we obtain the key estimate from here:
\begin{equation}\label{key}
\Delta_{\varphi}u\geq-C_{4.92}(n+\Delta\varphi)^{3n-3}u-C_{4.92}(n+\Delta\varphi)^{3n-3}.
\end{equation}
Next we plan to do iteration, using (\ref{key}).
Notice that for any $p>0$:
\begin{equation}
{1\over {2p+1}} \Delta_{\varphi}(u^{2p+1})= u^{2p}\Delta_{\varphi}u+
2p u^{2p-1}|\nabla_{\varphi}u|_{\varphi}^2.
\end{equation}
Integrate over $M$, we obtain:
\begin{equation}
\int_M2pu^{2p-1}|\nabla_{\varphi}u|_{\varphi}^2dvol_{\varphi}=-\int_Mu^{2p}
\Delta_{\varphi}udvol_{\varphi}.
\end{equation}
Plug in the key estimate (\ref{key}), we get:
\begin{equation}
\begin{split}
\int_M2pu^{2p-1}|\nabla_{\varphi}u|_{\varphi}^2dvol_{\varphi}&\leq C_{4.92}\int_M(n+\Delta\varphi)^{3n-3}
u^{2p+1}dvol_{\varphi}\\
&+C_{4.92}\int_M
(n+\Delta\varphi)^{3n-3}u^{2p}dvol_{\varphi}.
\end{split}
\end{equation}
Or equivalently:
\begin{equation}
\int_M|\nabla_{\varphi}(u^{p+\frac{1}{2}})|^2_{\varphi}dvol_{\varphi}\leq\frac{C_{4.92}(p+\frac{1}{2})^2}{2p}\int_M(n+\Delta\varphi)^{3n-3}(u^{2p+1}+u^{2p})dvol_{\varphi}.
\end{equation}
Observe that $u\geq K(n+\Delta\varphi)\geq e^{\frac{F}{n}}$, $u^{2p}\leq u^{2p+1}e^{-\frac{F}{n}} \leq C\cdot u^{2p+1}$. Hence for some constant $C_{3.91}$, we have
\begin{equation}
\int_M|\nabla_{\varphi}(u^{p+\frac{1}{2}})|_{\varphi}^2
dvol_{\varphi}\leq \frac{C_{4.93}(p+\frac{1}{2})^2}{2p}\int_M(n+\Delta\varphi)^{3n-3}u^{2p+1}dvol_{\varphi}.
\end{equation}
Since $dvol_{\varphi}=e^Fdvol_g$, and $F$ is bounded, we see
\begin{equation}\label{1.43}
\int_M|\nabla_{\varphi}(u^{p+\frac{1}{2}})|_{\varphi}^2
dvol_{g}\leq \frac{C_{4.94}(p+\frac{1}{2})^2}{2p}\int_M(n+\Delta\varphi)^{3n-3}u^{2p+1}dvol_{g}.
\end{equation}
Fix $\eps\in (0,2)$ to be determined, we estimate the right hand side of (\ref{1.43}):
\begin{equation}\label{1.44}
\int_M(n+\Delta\varphi)^{3n-3}u^{2p+1}dvol_g\leq
\bigg(\int_Mu^{(p+\frac{1}{2})(2+\eps)}dvol_g\bigg)^{\frac{2}{2+\eps}}\bigg(\int_M(n+\Delta\varphi)^{\frac{(3n-3)(2+\eps)}{\eps}}\bigg)^{\frac{\eps}{2+\eps}}.
\end{equation}
Denote $v=u^{p+\frac{1}{2}}$, then (\ref{1.43}) now becomes:
\begin{equation}\label{4.22new}
\int_M|\nabla_{\varphi}v|_{\varphi}^2dvol_g\leq\frac{C_{4.94}(p+\frac{1}{2})^2}{2p}\bigg(\int_M(n+\Delta\varphi)^{\frac{(3n-3)(2+\eps)}{\eps}}\bigg)^{\frac{\eps}{2+\eps}}\cdot\bigg(\int_Mv^{2+\eps}dvol_g\bigg)^{\frac{2}{2+\eps}}.
\end{equation}
We estimate the left hand side of (\ref{4.22new}) from below:
\begin{equation}
\begin{split}
|\nabla v|^{2-\eps}&\leq\sum_i|v_i|^{2-\eps}=\sum_i\frac{|v_i|^{2-\eps}}{(1+\varphi_{i\bar{i}})^{\frac{2-\eps}{2}}}\cdot(1+\varphi_{i\bar{i}})^{\frac{2-\eps}{2}}\\
&\leq\bigg(\sum_i\frac{|v_i|^2}{1+\varphi_{i\bar{i}}}\bigg)^{\frac{2-\eps}{2}}(\sum_i(1+\varphi_{i\bar{i}})^{\frac{2-\eps}{\eps}}\bigg)^{\frac{\eps}{2}}\leq |\nabla_{\varphi}v|_{\varphi}^{2-\eps}n^{\frac{\eps}{2}}(n+\Delta\varphi)^{\frac{2-\eps}{2}}.
\end{split}
\end{equation}
Integrate and use Holder inequality, we get:
\begin{equation}
\begin{split}
\int_M&|\nabla v|^{2-\eps}dvol_g\leq n^{\frac{\eps}{2}}\int_M|\nabla_{\varphi}v|_{\varphi}^{2-\eps}(n+\Delta\varphi)^{\frac{2-\eps}{2}}dvol_g\\
&\leq n^{\frac{\eps}{2}}\bigg(\int_M|\nabla_{\varphi}v|_{\varphi}^2\bigg)^{\frac{2-\eps}{2}}\bigg(\int_M(n+\Delta\varphi)^{\frac{2-\eps}{\eps}}dvol_g\bigg)^{\frac{\eps}{2}}.
\end{split}
\end{equation}
Therefore, for $p\geq \frac{1}{2}$, we may apply (\ref{4.22new}) to get:
\begin{equation}\label{1.47}
\begin{split}
\bigg(\int_M|\nabla v|^{2-\eps}dvol_g\bigg)^{\frac{2}{2-\eps}}   &\leq n^{\frac{\eps}{2-\eps}}\bigg(\int_M(n+\Delta\varphi)^{\frac{2-\eps}{\eps}}dvol_g\bigg)^{\frac{\eps}{2-\eps}}  \int_M|\nabla_{\varphi}v|_{\varphi}^2dvol_g\\
&\leq C_{4.95}pK_{\eps}\bigg(\int_Mv^{2+\eps}dvol_g\bigg)^{\frac{2}{2+\eps}}.
\end{split}
\end{equation}
Here
\begin{equation}\label{3.28new}
K_{\eps}= n^{\frac{\eps}{2-\eps}} \cdot \bigg(\int_M(n+\Delta\varphi)^{\frac{2-\eps}{\eps}}dvol_g\bigg)^{\frac{\eps}{2-\eps}}\cdot\bigg(\int_M(n+\Delta\varphi)^{\frac{(3n-3)(2+\eps)}{\eps}}\bigg)^{\frac{\eps}{2+\eps}}.
\end{equation}
Apply the Sobolev embedding with exponent $2-\eps$, and denote $\theta=\frac{2n(2-\eps)}{2n-2+\eps}$ to be the improved integrability, we get
$$
||v||_{L^{\theta}(dvol_g)}\leq C_{sob}(||\nabla v||_{L^{2-\eps}(dvol_g)}+||v||_{L^{2-\eps}(dvol_g)}).
$$
Recall that $v=u^{p+\frac{1}{2}}$, this means:
\begin{equation}\label{1.49}
\begin{split}
\bigg(\int_Mu^{(p+\frac{1}{2})\theta}dvol_g&\bigg)^{\frac{2}{\theta}}\leq C_{sob}\bigg(\big(\int_M|\nabla(u^{p+\frac{1}{2}})|^{2-\eps}dvol_g\big)^{\frac{2}{2-\eps}}+\big(\int_Mu^{(p+\frac{1}{2})(2-\eps)}dvol_g\big)^{\frac{2}{2-\eps}}\bigg)\\
&\leq C_{sob}\bigg(C_{4.95}pK_{\eps}\big(\int_Mu^{(p+\frac{1}{2})(2+\eps)}dvol_g\big)^{\frac{2}{2+\eps}}+\big(\int_Mu^{(p+\frac{1}{2})(2-\eps)}dvol_g\big)^{\frac{2}{2-\eps}}\bigg)\\
&\leq C_{4.96,\eps}p\big(\int_Mu^{(p+\frac{1}{2})(2+\eps)}dvol_g\big)^{\frac{2}{2+\eps}}.
\end{split}
\end{equation}
Here $C_{4.96,\eps}$ has the same dependence as $C_i$'s above, but with additional dependence on $\eps$.
From the 1st line to 2nd line, we used (\ref{1.47}).
 Now choose $\eps>0$ small so that $\theta>2+\eps$, then above estimate indeed improves integrability, namely we need
\begin{equation}\label{3.30new}
\frac{2n(2-\eps)}{2n-2+\eps}>2+\eps.
\end{equation}
We fix $\eps$ and (\ref{1.49}) gives for $p\geq\frac{1}{2}$:
\begin{equation}
||u||_{L^{(p+\frac{1}{2})\theta}}\leq\big(C_{4.97}p\big)^{\frac{1}{p+\frac{1}{2}}}||u||_{L^{(p+\frac{1}{2})(2+\eps)}}.
\end{equation}
Denote $\chi=\frac{\theta}{2+\eps}>1$, and choose $p+\frac{1}{2}=\chi^i$, for $i\geq0$. Then we obtain:
\begin{equation}
||u||_{L^{(2+\eps)\chi^{i+1}}}\leq\big(C_{4.97}\chi^i\big)^{\frac{1}{\chi^i}}||u||_{L^{(2+\eps)\chi^i}}.
\end{equation}
It follows that
\begin{equation}
||u||_{L^{\infty}}\leq C_{4.97}^{\sum_{i\geq0}\frac{1}{\chi^i}}\cdot\chi^{\sum_{i\geq0}\frac{i}{\chi^i}}||u||_{L^{2+\eps}}\leq C_{4.97}^{\sum_{i\geq0}\frac{1}{\chi^i}}\cdot\chi^{\sum_{i\geq0}\frac{i}{\chi^i}}||u||_{L^1}^{\frac{1}{2+\eps}}||u||_{L^{\infty}}^{\frac{1+\eps}{2+\eps}}.
\end{equation}
From above we get estimate of $||u||_{L^{\infty}}$ in terms of $||u||_{L^1}$.
But recall $u=e^{\frac{1}{2}F}|\nabla_{\varphi}F|^2+K(n+\Delta\varphi)$, so $L^1$ estimate is available.

Indeed, it is clear that $n+\Delta\varphi\in L^1$. To see $e^{\frac{1}{2}F}|\nabla_{\varphi}F|_{\varphi}^2\in L^1$, we just need to show $|\nabla_{\varphi}F|_{\varphi}^2\in L^1$ since $F$ is now assumed to be bounded.
Then we can calculate:
\begin{equation}
\Delta_{\varphi}(F^2)=2|\nabla_{\varphi}F|_{\varphi}^2+2F\Delta_{\varphi}F=2|\nabla_{\varphi}F|_{\varphi}^2+2F(-\underline{R}+tr_{\varphi}Ric).
\end{equation}
Integrate with respect to $dvol_{\varphi}=e^Fdvol_g$, we see
\begin{equation}
\begin{split}
\int_Me^{F}&|\nabla_{\varphi}F|_{\varphi}^2dvol_g=\int_Me^{F}F(\underline{R}-tr_{\varphi}Ric)dvol_{g}\leq C_{4.98}\int_M(1+tr_{\varphi}g)dvol_{g}\\
&\leq C_{3.96}(n+1)vol(M).
\end{split}
\end{equation}
Here $C_{4.98}$ may depend on $||F||_0$.
To see the range of $p_n$ asserted in the Remark \ref{r3.2}, we notice the choice of $\eps=\frac{1}{2n}$ verifies the requirement in (\ref{3.30new}).
With this choice, the highest power of $n+\Delta\varphi$ appearing in (\ref{3.28new}) is exactly $(3n-3)(4n+1)$. Once we have control over $K_{\eps}$, the rest of the proof goes through.
\end{proof}

\section{Entropy bound of the volume ratio and $C^0$ bound of K\"ahler potential }
The main goal of this section is to show the $C^0$ bound of $\varphi$ implies a bound for $\int_Me^FFdvol_g$ and vice versa:
\begin{thm} \label{Thm5.1} Let $(\varphi, F)$ be a smooth solution to cscK, then $\int_Me^FFdvol_g$ can be bounded in terms of $||\varphi||_0$. Conversely, a bound for $\int_Me^FFdvol_g$ implies a bound for $||F||_0$, in particular $||\varphi||_0$.
\end{thm}
The most difficult part of above theorem is to show that an upper bound for $\int_Me^FFdvol_g$ implies a bound on $||\varphi||_0$ and $||F||_0$, which is the main focus of this section. That $||\varphi||_0$ implies a bound for $\int_Me^FFdvol_g$ essentially follows from the fact that cscK are minimizers of $K$-energy. In particular, having a bound on $||\varphi||_0$ is enough to control $||F||_0$, hence estimates up to $C^{1,1}$, thanks to the results obtained in previous sections.
Actually we will see it is enough to have a bound for $\int_Me^F\Phi(F)dvol_g$, where $ \Phi(F)> 0 $ is coercive in $F$ in the sense that
\begin{enumerate}
\item  $\displaystyle \lim_{t\rightarrow -\infty }\;e^t \cdot \Phi(t) =  0$ and    $\displaystyle \lim_{t\rightarrow \infty}\;\Phi(t) = \infty$ 
\item $\displaystyle \lim_{t\rightarrow \infty} \;{{\Phi(t)}\over t} < \infty.\;$  
\end{enumerate}

We want to show that, under these conditions, an upper bound for $\int_Me^F\Phi(F)dvol_g$ will imply a bound for $\int_Me^{qF}dvol_g$ for any $q<\infty$.
This bound can then imply a bound for $||\varphi||_0$, due to the deep result by Kolodziej, \cite{Kolo98}, but an elementary argument which only uses Alexandrov maximum principle (Lemma \ref{abp}) and avoids pluripotential theory is also possible. 
This argument is due to Blocki (c.f. \cite{Blocki}).
From Corollary \ref{c5.4}, we obtain a bound for $||e^F||_0$. We have also shown in Proposition \ref{p2.1} that a $C^0$ bound of $\varphi$ will imply a lower bound for $F$.
Hence a bound for $||F||_0$ can be obtained this way. Then estimates in previous sections can be applied to obtain higher derivatives bound.

Define 
\begin{equation}
P(M,g)=\{\phi\in C^2(M,\mathbb{R}):g_{i\bar{j}}+\frac{\partial^2\phi}{\partial z_i\partial\bar{z}_j}\geq0,\,\sup_M\phi=0\}.
\end{equation}

The following result of Tian is well-known, whose proof may be found in \cite{tian87}, Proposition 2.1:
\begin{prop}\label{p5.2}
There exists two positive constant $\alpha$, $C_5$, depending only on $(M,g)$, such that
\begin{equation}
\int_Me^{-\alpha\phi}dvol_g\leq C_5\textrm{, for any $\phi\in P(M,g)$.}
\end{equation}
\end{prop}
Here $\alpha = \alpha(M, [\omega])$ is the so called $\alpha$-invariant. To start, we normalize $\varphi$ so that $\sup_M\varphi=0$.
We also need to consider the auxiliary K\"ahler potential $\psi\in \mathcal H$, which solves the following problem:

\begin{align}
\label{eq1}
&\det(g_{i\bar{j}}+\psi_{i\bar{j}})=\frac{e^F\Phi(F)\det(g_{i\bar{j}})}{\int_Me^F\Phi(F)dvol_g},\\
\label{sup}
&\sup_M\psi=0.
\end{align}
The existence of such $\psi$ follows from Yau's celebrated theorem on Calabi's volume conjecture (c.f. \cite{Yau78}, Theorem 2) .  Because of Proposition \ref{p5.2}, we know that $$\int_Me^{-\alpha\varphi}dvol_g\leq C_5,\qquad \int_Me^{-\alpha\psi}dvol_g\leq C_5.$$

We will show that the following estimate holds:
\begin{thm}\label{t5.2}
Given any $0<\eps<1$, there exists a constant $C_{5.1}$, depending on $\eps$, the background metric $g$, the choice of $\Phi$, and the bound $\int_Me^F\Phi(F)dvol_g$, such that
\begin{equation}
F+\eps\psi-2(1+\max_M|Ric|)\varphi\leq C_{5.1}.
\end{equation} 
\end{thm}
\begin{cor}\label{c5.3}
For any $0<q<\infty$, there exists a constant $C_{5.2}$, depending only on the background metric $g$, the choice of $\Phi$, the bound $\int_Me^F\Phi(F)dvol_g$, and $q$, such that 
\begin{equation}
\int_Me^{qF}dvol_g\leq C_{5.2}, \textrm{ $||\varphi||_0\leq C_{5.2}$, $||\psi||_0\leq C_{5.2}.$}
\end{equation}
\end{cor}

We will show this important corollary first. 

\begin{proof}
First we derive the estimate for $\int_Me^{qF}dvol_g$ with $q>1$.

From Theorem \ref{t5.2}, we know 
\begin{equation}
-\alpha\psi\geq\frac{\alpha}{\eps}\big(F-2(1+\max_M|Ric|)\varphi-C_{5.1}\big).
\end{equation}
hence
\begin{equation}
\begin{split}
C_5\geq \int_Me^{-\alpha\psi}dvol_g\geq&\int_M \exp\big(\frac{\alpha}{\eps}(F-2(1+\max_M|Ric|)\varphi-C_{5.1})\big)dvol_g\\
&\geq\int_M\exp\big(\frac{\alpha}{\eps}(F-C_{5.1})\big)dvol_g.
\end{split}
\end{equation}
The last inequality holds because we normalized $\varphi$ so that $\varphi\leq0$. Choose $\eps=\frac{\alpha}{q}$, then we immediately get the desired estimate for $\int_Me^{qF}dvol_g$.
The claimed estimate for $\varphi$ and $\psi$ immediately follows from the estimate for $||e^F||_{L^q} (q>2) $, given in the lemma below.   \end{proof}

\begin{lem} \label{lem5.3}Let  $\phi \in P(M,g)$ be such that $e^F = {\omega_\phi^n \over \omega_0^n} $ with  $e^F\in L^{2+s} (M, \omega_0),\;$ for some $s>0$. Then $||\phi||_0\leq C_{5.21}$, with $C_{5.21}$ depending only on the metric $\omega_0$, $s>0$ and $||e^F||_{L^{2+s}(M,\omega_0)}$.
\end{lem}
Note that this is a weaker result compared to  the famous theorem of Kolodziej \cite{Kolo98}, which shows $e^F\in L^{1+s}(M,\omega_0)$ is already sufficient. However, the weaker result as stated above can be proved in an elementary way using Alexandrov maximum principle, discovered by Blocki \cite{Blocki}.

Combining Theorem \ref{t5.2} and Corollary \ref{c5.3}, we immediately conclude:
\begin{cor}\label{c5.4}
There exists a constant $C_{5.2}$, depending only on the background metric $g$, the upper bound of $\int_Me^FFdvol_g$, such that 
$$
F\leq C_{5.2}.
$$
\end{cor}
\begin{proof}
Choose $\Phi(t)=\sqrt{t^2+1}$ and observe that $\int_Me^F\sqrt{F^2+1}dvol_g$ is controlled in terms of an upper bound of $\int_Me^FFdvol_g$. 
Then the result follows from Theorem \ref{t5.2} and Corollary \ref{c5.3}.
\end{proof}
Now let's prove Theorem \ref{t5.2}.
\begin{proof}
(of Theorem \ref{t5.2}) Let $0<\eps<1$ be given and fixed.
Let $d_0$ be chosen as in the proof of Corollary \ref{c5.3}.
 For any $p\in M$, let $\eta_p:M\rightarrow\mathbb{R}_+$ be a cut-off function such that $\eta_p(p)=1$, $\eta_p\equiv1-\theta$ outside the ball $B_{\frac{d_0}{2}}(p)$, 
with the estimate $|\nabla\eta_p|^2\leq\frac{4\theta^2}{d_0^2}$, $|\nabla^2\eta_p|\leq\frac{4\theta}{d_0^2}$.
Here $0<\theta<1$ is to be determined later.
Let $\delta>0$, $\lambda>0$ be constants to be determined.
Assume the function $e^{\delta(F+\eps\psi-\lambda\varphi)}$ achieves maximum at $p_0\in M$.
We now compute
\begin{equation}\label{5.31}
\begin{split}
&\Delta_{\varphi}\big(e^{\delta(F+\eps\psi-\lambda\varphi)}\eta_{p_0}\big)\\
&=\Delta_{\varphi}(e^{\delta(F+\eps\psi-\lambda\varphi)})\eta_{p_0}+e^{\delta(F+\eps\psi-\lambda\varphi)}\Delta_{\varphi}(\eta_{p_0})+e^{\delta(F+\eps\psi-\lambda\varphi)}2\delta\nabla_{\varphi}(F+\eps\psi-\lambda\varphi)\cdot\nabla_{\varphi}\eta_{p_0}\\
&=e^{\delta(F+\eps\psi-\lambda\varphi)}\eta_{p_0}\big(\delta^2|\nabla_{\varphi}(F+\eps\psi-\lambda\varphi)|_{\varphi}^2+\delta\Delta_{\varphi}(F+\eps\psi-\lambda\varphi)\big)\\
&\qquad\qquad\qquad+e^{\delta(F+\eps\psi-\lambda\varphi)}\Delta_{\varphi}(\eta_{p_0})+e^{\delta(F+\eps\psi-\lambda\varphi)}2\delta\nabla_{\varphi}(F+\eps\psi-\lambda\varphi)\cdot\nabla_{\varphi}\eta_{p_0}.
\end{split}
\end{equation}
First we can estimate
\begin{equation}\label{5.32}
\begin{split}
e^{\delta(F+\eps\psi-\lambda\varphi)}&\Delta_{\varphi}\eta_{p_0}\geq-e^{\delta(F+\eps\psi-\lambda\varphi)}|\nabla^2\eta_{p_0}|tr_{\varphi}g\\
&\geq -e^{\delta(F+\eps\psi-\lambda\varphi)}\frac{4\theta}{d_0^2(1-\theta)}\eta_{p_0}tr_{\varphi}g.
\end{split}
\end{equation}
\begin{equation}\label{5.33}
\begin{split}
&2\delta\nabla_{\varphi}(F+\eps\psi-\lambda\varphi)\cdot\nabla_{\varphi}\eta_{p_0}\geq -\delta^2\eta_{p_0}|\nabla_{\varphi}(F+\eps\psi-\lambda\varphi)|^2_{\varphi}-\frac{|\nabla_{\varphi}\eta_{p_0}|_{\varphi}^2}{\eta_{p_0}}\\
&\geq-\delta^2\eta_{p_0}|\nabla_{\varphi}(F+\eps\psi-\lambda\varphi)|_{\varphi}^2-\frac{|\nabla\eta_{p_0}|^2tr_{\varphi}g}{\eta_{p_0}}\\
&\geq-\delta^2\eta_{p_0}|\nabla_{\varphi}(F+\eps\psi-\lambda\varphi)|_{\varphi}^2-\frac{4\theta^2 tr_{\varphi}g}{d_0^2(1-\theta)}.
\end{split}
\end{equation}
Finally we compute
\begin{equation}\label{5.34}
\begin{split}
\Delta_{\varphi}(&F+\eps\psi-\lambda\varphi)=-(\underline{R}+\lambda n)+tr_{\varphi}Ric+\lambda tr_{\varphi}g+\eps\Delta_{\varphi}\psi\\
&\geq(-\underline{R}-\lambda n+\eps nA_{\Phi}^{-\frac{1}{n}}\Phi^{\frac{1}{n}}(F))+(\lambda-\eps-|Ric|)tr_{\varphi}g.
\end{split}
\end{equation}
Here $A_{\Phi}=\int_Me^F\Phi(F)dvol_g$.
In the above calculation, we noticed that
\begin{equation*}
\begin{split}
\Delta_{\varphi}\psi=g_{\varphi}^{i\bar{j}}&(g_{i\bar{j}}+\psi_{i\bar{j}})-tr_{\varphi}g\geq n\big(\det( g_{\varphi}^{i\bar{j}})\det(g_{i\bar{j}}+\psi_{i\bar{j}})\big)^{\frac{1}{n}}-tr_{\varphi}g\\
&= n(e^{-F}e^F\Phi(F)A_{\Phi}^{-1})^{\frac{1}{n}}-tr_{\varphi}g.
\end{split}
\end{equation*}
Plug (\ref{5.32}), (\ref{5.33}), (\ref{5.34}) back into (\ref{5.31}), we see
\begin{equation}\label{5.35}
\begin{split}
\Delta_{\varphi}&\big(e^{\delta(F+\eps\psi-\lambda\varphi)}\eta_{p_0}\big)\geq\delta\eta_{p_0} e^{\delta(F+\eps\psi-\lambda\varphi)}  (-\underline{R}-\lambda n+\eps nA_{\Phi}^{-\frac{1}{n}}\Phi^{\frac{1}{n}}(F))\\
&
-e^{\delta(F+\eps\psi-\lambda\varphi)} \big(\delta\eta_{p_0}(\lambda-\eps-|Ric|)-\frac{4\theta}{d_0^2(1-\theta)}\eta_{p_0}-\frac{4\theta^2}{d_0^2(1-\theta)^2}\big)tr_{\varphi}g.
\end{split}
\end{equation}
Now we choose various constants $\delta$, $\lambda$ and $\theta$ appearing above.

Since $0<\eps<1$, first we choose $\lambda=2(1+\max_M|Ric|)$.
Then we fix $\lambda$, and choose $\delta$ to be $2n\delta \lambda=\alpha$. We need to make sure the coefficient in front of $tr_{\varphi}g$ to be positive. This can be achieved by choosing $\theta$ to be sufficiently small. Indeed, with above choice of $\delta$ and $\lambda$, we may calculate:
\begin{equation}
\begin{split}
&\delta\eta_{p_0}(\lambda-\eps-|Ric|)-\frac{4\theta\eta_{p_0}}{d_0^2(1-\theta)}-\frac{4\theta^2}{d_0^2(1-\theta)^2}\\
&\geq\frac{1}{2}\delta(1-\theta)\lambda-\frac{4\theta\eta_{p_0}}{d_0^2(1-\theta)}-\frac{4\theta^2}{d_0^2(1-\theta)^2}\geq\frac{(1-\theta)\alpha}{4n}-\frac{4\theta}{d_0^2(1-\theta)}-\frac{4\theta^2}{d_0^2(1-\theta)^2}.
\end{split}
\end{equation}
Hence if we choose $\theta$ small enough, above $\geq0$.
After we made all the choices of $\delta$, $\lambda$, $\theta$, we obtain from (\ref{5.35}) that
\begin{equation}
\Delta_{\varphi}\big(e^{\delta(F+\eps\psi-\lambda\varphi)}\eta_{p_0})\geq\delta\eta_{p_0}e^{\delta(F+\eps\psi-\lambda\varphi)}(-\underline{R}-\lambda n+\eps nA_{\Phi}^{-\frac{1}{n}}\Phi^{\frac{1}{n}}(F)).
\end{equation}
Denote $u=e^{\delta(F+\eps\psi-\lambda\varphi)}$. Now we are ready to apply Alexandroff estimate in $B_{d_0}(p_0)$:
\begin{equation}\label{5.38}
\begin{split}
\sup_{B_{d_0}(p_0)}&u\eta_{p_0}\leq\sup_{\partial B_{d_0}(p_0)}u\eta_{p_0}\\
&+C_nd_0\bigg(\int_{B_{d_0}(p_0)}\frac{u^{2n}\big((-\underline{R}-\lambda n+\eps nA_{\Phi}^{-\frac{1}{n}}\Phi^{\frac{1}{n}}(F))^-\big)^{2n}}{e^{-2F}}dvol_g\bigg)^{\frac{1}{2n}}.
\end{split}
\end{equation}
We want to claim the integral appearing on the right hand side is bounded.
Indeed, the function been integrated is nonzero only if
$$
-\underline{R}-\lambda n+\eps nA_{\Phi}^{-\frac{1}{n}}\Phi^{\frac{1}{n}}(F)<0.
$$
By the coercivity of $\Phi$, this will imply an upper bound for $F$, say $F\leq C_{5.3}$, where the constant $C_{5.3}$ depends on $\eps$, the choice of $\Phi$, the integral bound $A_{\Phi}$, and the background metric $g$.
With this observation, we see
\begin{equation}
\begin{split}
&\int_{B_{d_0}(p_0)}\frac{u^{2n}\big((-\underline{R}-\lambda n+\eps nA_{\Phi}^{-\frac{1}{n}}\Phi^{\frac{1}{n}}(F))^-\big)^{2n}}{e^{-2F}}dvol_g\\
&\leq \int_{B_{d_0}(p_0)\cap\{F\leq C_{5.3}\}}e^{2n\delta(F+\eps\psi-\lambda\varphi)}e^{2F}(|\underline{R}|+\lambda n)^{2n}dvol_g\\
&\leq (|\underline{R}|+\lambda n)^{2n}e^{(2n\delta+2)C_{5.3}}\int_{B_{d_0}(p_0)}e^{2n\delta\eps\psi-2n\delta \lambda\varphi}dvol_g.
\end{split}
\end{equation}
But recall $\psi\leq 0$, and $2n\delta \lambda=\alpha$, we know 
\begin{equation}
\int_{B_{d_0}(p_0)}e^{2n\delta\eps\psi-2n\delta \lambda\varphi}dvol_g\leq\int_{B_{d_0}(p_0)}e^{-\alpha\varphi}dvol_g\leq C_{5.4}.
\end{equation}
Denote $I=(|\underline{R}|+\lambda n)^{2n}e^{(2n\delta+2)C_{5.3}}\int_{B_{d_0}(p_0)}e^{-\alpha\varphi}dvol_g$.
Now we go back to (\ref{5.38}) and obtain:
\begin{equation}
u(p_0)=\sup_Mu\leq (1-\theta)\sup_Mu+C_nd_0I^{\frac{1}{2n}}.
\end{equation}
Here we recall that $\eta_{p_0}\equiv1-\theta$ on $\partial B_{d_0}(p_0)$. This implies $\sup_Mu\leq \frac{C_nd_0I^{\frac{1}{2n}}}{\theta}$.
\end{proof}

\begin{lem}\label{abp}  {\bf Alexandroff maximum principle (c.f. \cite{GT}, Lemma 9.3})\\
Let $\Omega\subset\bR^d$ be a bounded domain. Suppose $u\in C^2(\Omega)\cap C(\bar{\Omega})$.
Denote $M=\sup_{\Omega}u-\sup_{\partial\Omega }u$.
Define
\begin{equation}\label{5.1nn}
\begin{split}
\Gamma^-(u,\Omega)=\{&x\in\Omega:u(y)\leq u(x)+\nabla u(x)\cdot (y-x),\\
&\quad\quad\quad\quad\,\,\textrm{for any $y\in\Omega$ and } |\nabla u(x)|\leq \frac{M}{3diam \Omega}\}.
\end{split}
\end{equation}
Then for some dimensional constant $C_d>0$:
$$
M\leq C_d\bigg(\int_{\Gamma^-(u,\Omega)}\det(-D^2u)dx\bigg)^{\frac{1}{d}}.
$$
In particular, suppose 
 $u$ satisfies $a_{ij}\partial_{ij}u\geq f$.
Here $a_{ij}$ satisfies the ellipticity condition $a_{ij}\xi_i\xi_j\geq0$.
Define $D^*=(\det a_{ij})^{\frac{1}{d}}$.
Then the following estimate holds:
\begin{equation}\label{2.61}
M\leq C_d'\,diam\,\Omega||\frac{f^-}{D^*}||_{L^d(\Omega)}.
\end{equation}
Here $C_d'$ is another dimensional constant. 
\end{lem}
\begin{rem} In this section (Proof of  Theorem \ref{t5.2} and Lemma \ref{lem5.3}), we apply this estimate with $d=2n$ to the operator $\Delta_{\varphi}$. After rewriting $\Delta_{\varphi}$ in terms of real coefficients, one can find $D^*=\big(\det (g_{\varphi})_{i\bar{j}}\big)^{-\frac{1}{n}}=e^{-\frac{F}{n}}\big(\det g_{i\bar{j}}\big)^{-\frac{1}{n}}$.
\end{rem}


Finally, we want to give a proof to Theorem \ref{Thm5.1}.
\begin{proof} It is well known that in a given K\"ahler class, cscK metrics
is global minimizer of the K-energy functional, by the main result of \cite{Ber14-01}. In particular, it follows that the K energy functional of $\varphi$ is a priori bounded from above.
Recall the decomposition formula  for K energy functional $E$, proved in \cite{chen00}: 
\begin{equation}\label{5.30}
K(\varphi)=\int_M\log\frac{\omega_{\varphi}^n}{\omega_0^n}\frac{\omega_{\varphi}^n}{n!}+J_{-Ric}(\varphi).
\end{equation}
In the above, $J_{-Ric}$ is defined in terms of its derivative, namely 
$$
\frac{dJ_{-Ric}}{dt}=\int_M\frac{\partial\varphi}{\partial t}(-tr_{\varphi}Ric+\underline{R})\frac{\omega_{\varphi}^n}{n!}.
$$
It is well known in the literature that $J_{-Ric}$ can be bounded in terms of $C^0$ norm of the potential function $\varphi.\;$
A bound for $\int_Me^F|F|dvol_g$ follows from here. 

Now we prove the second part of the theorem. First Corollary \ref{c5.4} gives a bound for $F$ from above and Corollary \ref{c5.3} gives a bound for $||\varphi||_0$. Proposition \ref{p2.1} gives a bound for $F$ from below.
\end{proof}

\section{Some local estimates}
In this section, we show some localized version of our previous estimates.
Suppose we have a solution $\varphi$ to (\ref{csck1}), (\ref{csck2}) in the unit ball $B_1(0)\subset\mathbb{C}^n$.
First we can find a potential $\rho$ to the background metric $g$, namely 
$$
\omega_0=\sqrt{-1}\partial\bar{\partial}\rho,\textrm{ in $B_1(0)$.}
$$
Denote $\phi=\varphi+\rho$, $G=F+\log\det g_{i\bar{j}}$, then the equation (\ref{csck1}), (\ref{csck2}) can be rewritten as:
\begin{align}
\label{csck-1}
&\det\phi_{i\bar{j}}=e^G,\\
\label{csck-2}
&\Delta_{\phi}G=-\underline{R}.
\end{align}
In the above, $\Delta_{\phi}=\phi^{i\bar{j}}\partial_{i\bar{j}}$.
In the following, we show that if $\Delta\phi\in L^p(B_1(0))$, and $\sum_i\frac{1}{\phi_{i\bar{i}}}\in L^p(B_1(0))$ for $p$ sufficiently large depending only on dimension $n$, then we have $\frac{1}{C}\leq\phi_{i\bar{j}}\leq C$, for some constant $C$. More precisely, 
\begin{prop}\label{p6.1}
Let $\phi$ be a smooth pluri-subharmonic solution to (\ref{csck-1}), (\ref{csck-2}) in $B_1(0)\subset\mathbb{C}^n$, such that $\Delta\phi\in L^p(B_1(0))$ and $\sum_i\frac{1}{\phi_{i\bar{i}}}\in L^p(B_1(0))$ for some $p>3n(n-1)$. Then there exists a constant $C_6$, depending only on $p$, $||\Delta\phi||_{L^p(B_1(0))}$, $||\sum_i\frac{1}{\phi_{i\bar{i}}}||_{L^p(B_1(0))}$, such that
$\frac{1}{C_6}\leq\phi_{i\bar{j}}\leq C_6$, $|\nabla G|\leq C_6$ in $B_{\frac{1}{2}}(0)$. 
\end{prop}
By the same argument in (\ref{p1.1}), we have the following corollary:
\begin{cor}\label{c6.2}
Under the assumption of Proposition \ref{p6.1}, for any $0<\theta<1$, we have $||D^k\phi||_{0,B_{\theta}}\leq C(k)$, for any $k\geq2$. Here $C(k)$ has the same dependence as described in Proposition \ref{p6.1} besides dependence on $\theta$ and on $||\phi||_{0,B_1}$.
\end{cor}
Now we prove Proposition \ref{p6.1}, using Lemma \ref{l4.2} stated and proved later.
\begin{proof}
(of Proposition \ref{p6.1})First we want to get boundedness of $G$, using the second equation. We can write the second equation as
\begin{equation}
\det(\phi_{\alpha\bar{\beta}})\phi^{i\bar{j}}\partial_{i\bar{j}}G=e^GG.
\end{equation}
which is equivalent to:
\begin{equation}\label{6.4}
Re\big(\partial_i(\det(\phi_{\alpha\bar{\beta}})\phi^{i\bar{j}}\partial_{\bar{j}}G)\big)=e^GG.
\end{equation}
Denote $a_{i\bar{j}}=\det(\phi_{\alpha\bar{\beta}})\phi^{i\bar{j}}$, which is a hermitian matrix, then for some constant $c_n>0$, we have
\begin{equation}\label{6.5}
\frac{1}{c_n\big(\sum_i\frac{1}{\phi_{i\bar{i}}}\big)^{n-1}}I\leq a_{i\bar{j}}\leq c_n(\Delta\phi)^{n-1}I.
\end{equation}
The left hand side of (\ref{6.4}) is a real elliptic operator in divergence form, which satisfies an ellipticity condition same as (\ref{6.5}).
We wish to apply Lemma \ref{l4.2} to the equation (\ref{6.4}).
Using (\ref{6.5}), we can take $\lambda=\big(\Delta\phi+\sum_i\frac{1}{\phi_{i\bar{i}}}\big)^{n-1}$, and $f=e^GG$.
In order to apply Lemma \ref{l4.2}, we need to show $(\Delta\phi)^{n-1},\,(\sum_i\frac{1}{\phi_{i\bar{i}}})^{n-1}\in L^p(B_1)$, and $e^GG\in L^{p/2}(B_1)$ for some $p>3n$.
The desired integrability for $\Delta\phi$ and $\sum_i\frac{1}{\phi_{i\bar{i}}}$ is clear from assumption, while for $e^GG$, since $e^G\leq(\Delta\phi)^n$, we just need to make sure $(\Delta\phi)^n\in L^{p'}$ for some $p'>\frac{3n}{2}$. This is again clear from our assumption on $\Delta\phi$.
So we can apply Lemma \ref{l4.2} to conclude $G$ is bounded(with the said dependence) on any interior ball of $B_1$. In the following we assume $G$ is bounded on $B_1$ without loss of generality.

The estimate for $\Delta\phi$ is really similar to our calculation in section 4, so we will be suitably brief here.

Choose any point $p$ and we can do a unitary coordinate transform so that $\phi_{i\bar{j}}(p)=\phi_{i\bar{i}}(p)\delta_{ij}$. 
We can compute
\begin{equation}\label{6.6}
\Delta_{\phi}(|\nabla_{\phi}G|^2)=\frac{1}{\phi_{i\bar{i}}\phi_{\alpha\bar{\alpha}}}|\phi_{i\alpha}-\sum_p\frac{\phi_{i\alpha\bar{p}}G_p}{\phi_{p\bar{p}}}|^2+\frac{|G_{p\bar{i}}|^2}{\phi_{i\bar{i}}\phi_{p\bar{p}}}-\frac{G_{q\bar{p}}G_pG_{\bar{q}}}{\phi_{p\bar{p}}\phi_{q\bar{q}}}.
\end{equation}
Here $|\nabla_{\phi}G|^2=\phi^{p\bar{q}}G_pG_{\bar{q}}$.
\begin{equation}\label{6.7}
\Delta_{\phi}(e^{\frac{1}{2}G}|\nabla_{\phi}G|^2)=\Delta_{\phi}(e^{\frac{1}{2}G})|\nabla_{\phi}G|^2+e^{\frac{1}{2}G}\Delta_{\phi}(|\nabla_{\phi}G|^2)+\frac{1}{2}e^{\frac{1}{2}G}\frac{G_i(|\nabla_{\phi}G|^2)_{\bar{i}}+G_{\bar{i}}(|\nabla_{\phi}G|^2)_i}{\phi_{i\bar{i}}}.
\end{equation}
One can also compute 
\begin{equation}\label{6.8}
\frac{G_i(|\nabla_{\phi}G|^2)_{\bar{i}}}{\phi_{i\bar{i}}}=\frac{G_iG_p}{\phi_{i\bar{i}}\phi_{p\bar{p}}}(G_{\bar{p}\bar{i}}-\sum_t\frac{\phi_{t\bar{q}\bar{i}}G_{\bar{t}}}{\phi_{t\bar{t}}})+\frac{G_{p\bar{i}}G_{\bar{p}}G_i}{\phi_{p\bar{p}}\phi_{i\bar{i}}}.
\end{equation}
Combining (\ref{6.6}), (\ref{6.7}), (\ref{6.8}), we obtain
\begin{equation}
\Delta_{\phi}(e^{\frac{1}{2}G}|\nabla_{\phi}G|^2)e^{-\frac{1}{2}G}=-\frac{1}{2}\underline{R}|\nabla_{\phi}G|^2+\frac{1}{\phi_{i\bar{i}}\phi_{\alpha\bar{\alpha}}}|G_{i\alpha}-\sum_p\frac{\phi_{i\alpha\bar{p}}G_p}{\phi_{p\bar{p}}}-\frac{1}{2}\phi_i\phi_{\alpha}|^2+\frac{|G_{p\bar{i}}|^2}{\phi_{i\bar{i}}\phi_{p\bar{p}}}.
\end{equation}
Also we can compute
\begin{equation}
\Delta_{\phi}(\Delta\phi)=\frac{|\phi_{i\bar{j}p}|^2}{\phi_{i\bar{i}}\phi_{j\bar{j}}}+\Delta G.
\end{equation}
Hence
\begin{equation}
\begin{split}
\Delta_{\phi}&(e^{\frac{1}{2}G}|\nabla_{\phi}G|^2+\Delta\phi)\geq-\frac{1}{2}\underline{R}e^{\frac{1}{2}G}|\nabla_{\phi}G|^2+\frac{|G_{p\bar{i}}|^2e^{\frac{1}{2}G}}{\phi_{i\bar{i}}\phi_{p\bar{p}}}+\Delta G\\
&\geq -\frac{1}{2}\underline{R}e^{\frac{1}{2}G}|\nabla_{\phi}G|^2+\frac{e^{\frac{1}{2}G}G_{i\bar{i}}^2}{\phi_{i\bar{i}}^2}-\frac{1}{2}\frac{G_{i\bar{i}}^2e^{\frac{1}{2}G}}{\phi_{i\bar{i}}^2}-\frac{1}{2}(\Delta\phi)^2e^{-\frac{1}{2}G}.
\end{split}
\end{equation}
In the last inequality above, we noticed
$$
\Delta G=\sum_iG_{i\bar{i}}\leq\frac{1}{2}\frac{G_{i\bar{i}}^2e^{\frac{1}{2}G}}{\phi_{i\bar{i}}^2}+\frac{1}{2}\sum_i\phi_{i\bar{i}}^2e^{-\frac{1}{2}G}\leq\frac{1}{2}\frac{G_{i\bar{i}}^2e^{\frac{1}{2}G}}{\phi_{i\bar{i}}^2}+\frac{1}{2}(\Delta\phi)^2e^{-\frac{1}{2}G}.
$$
Denote $u=e^{\frac{1}{2}G}|\nabla_{\phi}G|^2+\Delta\phi$, then we know
\begin{equation}
\Delta_{\phi}(u)\geq-(\frac{1}{2}\underline{R}+\frac{1}{2}\Delta\phi e^{-\frac{1}{2}G})u.
\end{equation}
Denote $f=\frac{1}{2}\underline{R}+\frac{1}{2}\Delta\phi e^{-\frac{1}{2}G}$. Recall that we now already know $G$ is bounded. Our assumption implies $f\in L^p$ for some $p>3n$.
Hence we may invoke Lemma \ref{l4.2} to get the desired result.
\end{proof}

As a direct consequence of above argument, we can now prove Corollary \ref{c1.3new}.
\begin{proof}
(of Corollary \ref{c1.3new})
Define $F=\log\det u_{\alpha\bar{\beta}}$, first we show that $F$ is a constant.
By the assumption, we can take a sequence of $r_{s}\rightarrow\infty$, and a constant $M$, such that
\begin{equation}\label{6.13new}
\sup_{s\geq1}\frac{1}{r_{s}^{2n}}\int_{B_{r_{s}}(0)}(\Delta u)^p+\big(\sum_k\frac{1}{u_{k\bar{k}}}\big)^p\leq M.
\end{equation}
Define $u_{s}(z)=\frac{1}{r_{s}^2}u(r_s z)$.
Let $\Delta_{s}:=u_{s}^{i\bar{j}}\partial_{i\bar{j}}$, the Laplace operator in $\mathbb{C}^n$ defined by the metric $\sqrt{-1}\partial\bar{\partial}u_{s}$.
Also we denote $F_{s}:=\log\det \partial_{\alpha\bar{\beta}}u_s$, then $F_s(z)=F(r_sz)$.
Hence $\Delta_sF_s=0$ in $B_1(0)$, and (\ref{6.13new}) implies 
\begin{equation}\label{6.14new}
\sup_{s\geq1}\int_{B_1}(\Delta u_s)^p+\big(\sum_k\frac{1}{(u_s)_{k\bar{k}}}\big)^p\leq M.
\end{equation}
Proposition \ref{p6.1} shows that there exists a positive constant $C_7$, independent of $s$, such that $\frac{1}{C_7}\leq (u_s)_{i\bar{j}}\leq C_7$  and $|\nabla F_s|\leq C_7$ in $B_{\frac{1}{2}}(0)$. Rescaling back, we find that $|\nabla F|\leq \frac{C_7}{r_s}$ in $B_{\frac{1}{2}r_s}(0)$. Sending $s\rightarrow\infty$, we get $\nabla F\equiv0$ on $\mathbb{C}^n$. Namely we have $\det \partial_{\alpha\bar{\beta}}u_s=c$ for some $c>0$ on $B_1$.
Then we may use Evans-Krylov theorem to conclude for some $\alpha>0$
$$
\sup_k[\sqrt{-1}\partial\bar{\partial}u_s]_{\alpha,B_{\frac{1}{4}}}\leq M_1.
$$
In terms of $u$, this implies 
$$
r_s^{\alpha}\frac{|u_{i\bar{j}}(z_1)-u_{i\bar{j}}(z_2)|}{|z_1-z_2|^{\alpha}}\leq M, \textrm{ for any $z_1,\,z_2\in B_{\frac{r_s}{4}}(0)$.}
$$
Letting $s\rightarrow\infty$, we obtain $u_{i\bar{j}}(z_1)=u_{i\bar{j}}(z_2)$, for any $z_1,\,z_2\in \mathbb{C}^n$. This implies the Levi Hessian of $u$ is constant.
\end{proof}
The following is a technical lemma which we used in the proof of Proposition \ref{p6.1}. The way to prove it is the standard Moser's iteration and may have existed in literature but we were not able to find the exactly reference, so we include a proof here.
\begin{lem}\label{l4.2}
Suppose $u\geq0$ satisfies in $B_1\subset\bR^d$:
$$
\partial_i\big(a^{ij}\partial_ju\big)\geq fu+g.
$$
Here $\frac{1}{\lambda(x)}\leq a^{ij}(x)\leq\lambda(x)$, with $\lambda(x)\in L^p(B_1)$, $f,\,g\in L^{p/2}(B_1)$ for some $p>\frac{3d}{2}$, then there exists a constant $C$, depending on $p$, $||\lambda||_{L^p(B_1)}$, $||f||_{L^{p/2}(B_1)}$, $||g||_{L^{p/2}(B_1)}$, such that
$$\sup_{B_{\frac{1}{2}}}u\leq C(||u||_{L^1(B_1)}+1).
$$
\end{lem}
\begin{proof}
The proof follows the same argument as the uniformly elliptic case.
From the inequality we know that for any $\zeta\in C_c^{\infty}(B_1)$, with $\zeta\geq0$, the following holds:
\begin{equation}\label{106}
\int_{B_1}a^{ij}\partial_ju\partial_i\zeta dx\leq\int_{B_1}fu\zeta+g\zeta dx.
\end{equation}
Now let $\eta\in C^{\infty}_c(B_1)$, define $\bar{u}=u+1$. Take $\zeta=\eta^2\bar{u}^{\beta}$, for some $\beta>0$.
We plug in this $\zeta$ and obtain
\begin{equation}
\begin{split}
\int_{B_1}&\beta a_{ij}\partial_i\bar{u}\partial_j\bar{u}\bar{u}^{\beta-1}\eta^2\leq\int-a_{ij}\partial_ju\partial_i\eta\bar{u}^{\beta}2\eta+|f|\bar{u}^{\beta+1}\eta^2+|g|\bar{u}^{\beta}\eta^2\\
&\leq\int_{B_1}\frac{\beta}{2}a_{ij}\partial_i\bar{u}\partial_j\bar{u}\bar{u}^{\beta-1}\eta^2+\frac{4}{\beta} a_{ij}\partial_i\eta\partial_j\eta\bar{u}^{\beta+1}+(|f|+|g|)\bar{u}^{\beta+1}\eta^2.
\end{split}
\end{equation}
Use the ellipticity condition to get:

\begin{equation}\label{108}
\int_{B_1}\frac{\beta}{\lambda}|\nabla\bar{u}|^2\bar{u}^{\beta-1}\eta^2\leq\int_{B_1}\big(\frac{4\lambda}{\beta}|\nabla\eta|^2+|f|\eta^2+|g|\eta^2\big)\bar{u}^{\beta+1}.
\end{equation}
This is equivalent to:
\begin{equation}\label{6.16}
\int_{B_1}|\nabla(\bar{u}^{\frac{\beta+1}{2}})|^2\eta^2\frac{1}{\lambda}\leq\frac{(\beta+1)^2}{\beta^2}\int_{B_1}\lambda|\nabla\eta|^2\bar{u}^{\beta+1}+\frac{(\beta+1)^2}{4\beta}\int_{B_1}(|f|+|g|)\bar{u}^{\beta+1}\eta^2.
\end{equation}
Next observe
$$
|\nabla(\bar{u}^{\frac{\beta+1}{2}}\eta)|^2\leq 2|\nabla(\bar{u}^{\frac{\beta+1}{2}})|^2\eta^2+2\bar{u}^{\beta+1}|\nabla\eta|^2.
$$
Hence it follows from (\ref{6.16}) that if $\beta\geq1$,
\begin{equation}
\begin{split}
\int_{B_1}|\nabla(\bar{u}^{\frac{\beta+1}{2}}\eta)|^2\frac{1}{\lambda}&\leq\int_{B_1}\big(\frac{2\lambda(\beta+1)^2}{\beta^2}+\frac{2}{\lambda}\big)\bar{u}^{\beta+1}|\nabla\eta|^2+\frac{(\beta+1)^2}{2\beta}\int_{B_1}(|f|+|g|)\bar{u}^{\beta+1}\eta^2\\
&\leq\int_{B_1}10\lambda\bar{u}^{\beta+1}|\nabla\eta|^2+2\beta\int_{B_1}(|f|+|g|)\bar{u}^{\beta+1}\eta^2.
\end{split}
\end{equation}
We would like to get rid of the $\lambda$ in the above estimate. Let $\eps>0$ to be determined, then we have
\begin{equation}
||\nabla(\bar{u}^{\frac{\beta+1}{2}}\eta)||_{L^{2-\eps}}^2\leq ||\lambda||_{L^{\frac{2}{\eps}-1}}^{\frac{2-\eps}{2}}\int_{B_1}\frac{1}{\lambda}|\nabla(\bar{u}^{\frac{\beta+1}{2}}\eta)|^2.
\end{equation}
On the other hand, we estimate the right hand side by H\"older's inequality:
\begin{align}
\label{6.19}
&\int_{B_1}\lambda\bar{u}^{\beta+1}|\nabla\eta|^2\leq||\lambda||_{L^{\frac{2}{\eps}-1}}||\bar{u}^{\frac{\beta+1}{2}}|\nabla\eta|||_{L^{\frac{2-\eps}{1-\eps}}}^2,
\\
\label{6.20}
&\int_{B_1}(|f|+|g|)\bar{u}^{\beta+1}\eta^2\leq\big(||f||_{L^{p/2}}+||g||_{L^{p/2}}\big)||\bar{u}^{\frac{\beta+1}{2}}\eta||_{L^{\frac{2p}{p-2}}}^2.
\end{align}
Therefore,
\begin{equation}
||\nabla(\bar{u}^{\frac{\beta+1}{2}}\eta)||_{L^{2-\eps}}^2\leq ||\lambda||_{L^{\frac{2}{\eps}-1}}^{\frac{2-\eps}{2}}\big(10||\lambda||_{L^{\frac{2}{\eps}-1}}||\bar{u}^{\frac{\beta+1}{2}}|\nabla\eta|||^2_{L^{\frac{2-\eps}{1-\eps}}}
+2\beta(||f||_{L^{p/2}}+||g||_{L^{p/2}})||\bar{u}^{\frac{\beta+1}{2}}\eta||_{L^{\frac{2p}{p-2}}}^2\big).
\end{equation}
Now we choose $\eps=\frac{2}{p+1}$, then $\frac{2}{\eps}-1=p$. With this choice, we have $\frac{2-\eps}{1-\eps}<\frac{2p}{p-2}$ in the above, then we find for some constant $C_{6.1}$, depending on $||\lambda||_{L^p}$, $||f||_{L^{p/2}}$, $||g||_{L^{p/2}}$, such that
\begin{equation}
||\nabla(\bar{u}^{\frac{\beta+1}{2}}\eta)||_{L^{\frac{2p}{p+1}}}^2\leq C_{6.1}\big(||\bar{u}^{\frac{\beta+1}{2}}|\nabla\eta|||^2_{L^{\frac{2p}{p-2}}}+\beta||\bar{u}^{\frac{\beta+1}{2}}\eta||_{L^{\frac{2p}{p-2}}}^2\big).
\end{equation}
Fix $\frac{1}{2}\leq r<R\leq1$, Denote $r_i=r+2^{-i}(R-r)$, for $i\geq0$. 
Note that $r_0=R$, and $r_i\rightarrow r$ as $i\rightarrow\infty$.
We choose the cut-off function $\eta$ so that $0\leq\eta\leq1$, $\eta\equiv1$ on $B_{r_{i+1}}$, $supp\,\eta\subset B_{r_i}$, and $|\nabla\eta|\leq\frac{2}{r_i-r_{i+1}}=\frac{2^{i+2}}{R-r}$.
Denote $\theta$ to be such that $\frac{1}{\theta}=\frac{p+1}{2p}-\frac{1}{d}$.
Since $p>\frac{3d}{2}$, it follows that $\theta>\frac{2p}{p-2}$.
Then apply the Sobolev inequality to get
\begin{equation}
||\bar{u}^{\frac{\beta+1}{2}}||_{L^{\theta}(B_{r_{i+1}})}\leq \frac{C_{6.2}\sqrt{\beta}2^{i+2}}{R-r}||\bar{u}^{\frac{\beta+1}{2}}||_{L^{\frac{2p}{p-2}}(B_{r_i})}.
\end{equation}
This is equivalent to:
\begin{equation}\label{6.24}
||\bar{u}||_{L^{\frac{\theta(\beta+1)}{2}}(B_{r_{i+1}})}\leq\frac{C_{6.2}^{\frac{2}{\beta+1}}\beta^{\frac{1}{\beta+1}}2^{\frac{2(i+2)}{\beta+1}}}{(R-r)^{\frac{2}{\beta+1}}}||\bar{u}||_{L^{\frac{2p}{p-2}\cdot\frac{\beta+1}{2}}(B_{r_i})}.
\end{equation}
Now denote $\theta=\chi\cdot\frac{2p}{p-2}$ for some $\chi>1$, and choose $\beta$ to be $\frac{\beta+1}{2}=\chi^i$, then we obtain from (\ref{6.24}):
\begin{equation}
||\bar{u}||_{L^{\frac{2p\chi^{i+1}}{p-2}}(B_{r_{i+1}})}\leq \big(\frac{2C_{6.2}}{R-r}\big)^{\frac{1}{\chi^i}}\chi^{\frac{i}{\chi^i}}2^{\frac{i+2}{\chi^i}}||\bar{u}||_{L^{\frac{2p\chi^i}{p-2}}(B_{r_i})}, \textrm{ for $i\geq0$.}
\end{equation}
Iterating this inequality we obtain for any $\frac{1}{2}\leq r<R\leq 1$, and for some constant $C_{6.3}$ independent of $r$, $R$,
\begin{equation}
\begin{split}
&||\bar{u}||_{L^{\infty}(B_r)}\leq \frac{C_{6.3}}{(R-r)^{\sum_{i\geq0}\chi^{-i}}}||\bar{u}||_{L^{\frac{2p}{p-2}}(B_R)}=\frac{C_{6.3}}{(R-r)^{\frac{\chi}{\chi-1}}}||\bar{u}||_{L^{\frac{2p}{p-2}}(B_R)}\\
&\leq\frac{C_{6.3}}{(R-r)^{\frac{\chi}{\chi-1}}}||\bar{u}||_{L^1(B_R)}^{\frac{p-2}{2p}}||\bar{u}||_{L^{\infty}(B_R)}^{\frac{p+2}{2p}}\leq\frac{1}{2}||\bar{u}||_{L^{\infty}(B_R)}+\frac{2^{\frac{p+2}{p-2}}C_{6.3}^{\frac{2p}{p-2}}}{(R-r)^{\frac{2\chi p}{(\chi-1)(p-2)}}}||\bar{u}||_{L^1(B_R)}.
\end{split}
\end{equation}
The desired conclusion now follows from the following lemma applied to $f(r)=||\bar{u}||_{L^{\infty}(B_r)}$, which is a special case of Lemma 4.3 in \cite{H-L}.
\end{proof}
\begin{lem}
Let $f:[\frac{1}{2},1]\rightarrow\bR$ be nonnegative, monotone increasing, such that there exists $M>0$, $\alpha>0$, such that for any $\frac{1}{2}\leq r<R\leq 1$, it holds
$$
f(r)\leq \frac{1}{2}f(R)+\frac{M}{(R-r)^{\alpha}}.
$$
Then for some $C_{\alpha}>0$ depending only on $\alpha$, we have
$$
f(\frac{1}{2})\leq C_{\alpha}M.
$$
\end{lem}

Next we show that when $n=2$, for the solution to (\ref{csck-1}) and (\ref{csck-2}), $|\nabla\phi|$ locally bounded implies $G$ is locally bounded from above. More precisely,
\begin{prop}
Let $\phi$ be a smooth solution to (\ref{csck-1}), (\ref{csck-2}) in $B_1\subset\mathbb{C}^2$ such that $|\nabla\phi|$ is bounded. Then for some constant $C_{6.4}$, we have
\begin{equation}
e^G\leq C_{6.4}\textrm{ in $B_{\frac{1}{2}}$.}
\end{equation}
Here $C_{6.4}$ depends only on $||\nabla\phi||_0$ and $\underline{R}$.
\end{prop}
\begin{proof}
Let $0<\delta<1$ and $K>1$ to be determined. We will compute $\Delta_{\phi}(e^{\delta G}(|\nabla\phi|^2+K))$.
As before, for any point $p\in B_1$ we are considering, we can always do a unitary coordinate transform which makes $\phi_{i\bar{j}}(p)=\phi_{i\bar{i}}(p)\delta_{ij}$.
Under this coordinate, we can compute:
\begin{equation}
\begin{split}
&\Delta_{\phi}(e^{\delta G}(|\nabla\phi|^2+K))=e^{\delta G}(\delta^2|\nabla_{\phi}G|^2-\delta\underline{R})(|\nabla\phi|^2+K)+e^{\delta G}\Delta_{\phi}(|\nabla\phi|^2)\\
&+e^{\delta G}\delta\frac{G_i(|\nabla\phi|^2)_{\bar{i}}+G_{\bar{i}}(|\nabla\phi|^2)_i}{\phi_{i\bar{i}}}.
\end{split}
\end{equation}
Similar to the calculation in Theorem \ref{t2.1}, we can find:
\begin{align}
\label{6.35}
&\Delta_{\phi}(|\nabla\phi|^2)=\frac{|\phi_{ij}|^2}{\phi_{i\bar{i}}}+\Delta\phi+G_i\phi_{\bar{i}}+G_{\bar{i}}\phi_i\\
\label{6.36}
&(|\nabla\phi|^2)_i=\sum_j\phi_{ij}\phi_{\bar{j}}+\phi_{i\bar{i}}\phi_{\bar{i}}.
\end{align}
Hence we obtain
\begin{equation}\label{6.37}
\begin{split}
&\Delta_{\phi}(e^{\delta G}(|\nabla\phi|^2+K))
=\frac{e^{\delta G}}{\phi_{i\bar{i}}}|\delta\phi_j G_i+\phi_{ij}|^2+Ke^{\delta G}(\delta^2|\nabla_{\phi}G|^2-\delta\underline{R})-e^{\delta G}\delta\underline{R}|\nabla\phi|^2\\
&+e^{\delta G}\Delta\phi+e^{\delta G}(1+\delta)(G_i\phi_{\bar{i}}+G_{\bar{i}}\phi)\geq e^{\delta G}K\delta^2|\nabla_{\phi}G|^2+e^{\delta G}\Delta\phi-\delta\underline{R}e^{\delta G}|\nabla\phi|^2\\
&-\delta\underline{R}Ke^{\delta G}-\frac{1}{2}K\delta^2e^{\delta G}|\nabla_{\phi}G|^2-\frac{1}{2}\frac{(1+\delta)^2e^{\delta G}|\nabla\phi|^2\Delta\phi}{K\delta^2}.
\end{split}
\end{equation}
Now we choose $\delta=\frac{1}{8}$, and we choose $K$ sufficiently large so that $\frac{(1+\delta)^2|\nabla\phi|^2}{K\delta^2}<1$. Hence we obtain from (\ref{6.37}):
\begin{equation}
\Delta_{\phi}(e^{\delta G}(|\nabla\phi|^2+K))\geq \frac{1}{2}e^{\delta G}\Delta\phi-e^{\delta G}C_{6.5}.
\end{equation}
Here $C_9$ depends only on $\underline{R}$ and $||\nabla\phi||_0$. Define $\eta(z)=(1-|z|^2)^{-1}$ for $z\in B_1$. We show that $|\Delta_{\phi}\eta|\leq C_{6.6}\eta^3\sum_i\frac{1}{\phi_{i\bar{i}}}$.Indeed,
\begin{equation*}
\begin{split}
\Delta_{\phi}\eta&=\phi^{i\bar{j}}\partial_{i\bar{j}}(\eta)=\phi^{i\bar{j}}\bigg((1-|z|^2)^{-2}\delta_{ij}+2(1-|z|^2)^{-3}\bar{z}_iz_j\bigg)\\
&=\phi^{i\bar{j}}\eta^3\big((1-|z|^2)\delta_{ij}+2\bar{z}_iz_j\big).
\end{split}
\end{equation*}
From this the claim follows easily.
Denote $v=e^{\delta G}(|\nabla\phi|^2+K)$. Suppose the function $v-\eta$ achieves maximum at $p\in B_1$. There are two possibilities:

Suppose $v(p)-\eta(p)\leq0$, then we immediately conclude that 
$$
v(z)\leq \eta(z)\leq\frac{4}{3},\textrm{ for any $z\in B_{\frac{1}{2}}$.}
$$
Then we are done.

Suppose otherwise $v(p)-\eta(p)\geq0$, then we know at $p$:
\begin{equation}
\begin{split}
0&\geq\Delta_{\phi}(v-\eta)(p)\geq\frac{1}{2}e^{\delta G}\Delta\phi-e^{\delta G}C_{6.5}-C_{6.6}\eta^3\sum_i\frac{1}{\phi_{i\bar{i}}}\\&\geq\frac{1}{2}e^{\delta G}\Delta\phi-C_{6.6}v^3e^{-G}\Delta\phi-e^{\delta G}C_{6.5}\geq\frac{1}{2}e^{\delta G}\Delta\phi-C_{6.7}e^{-(1-3\delta)G}\Delta\phi-e^{\delta G}C_{6.5}.
\end{split}
\end{equation}
In the third inequality above, we used that $\sum_i\frac{1}{\phi_{i\bar{i}}}=e^{-G}\Delta\phi$, which is true only in dimension 2.
Also we used that at $p$, $\eta\leq v$.

Suppose at $p$, we have $\frac{1}{4}e^{\delta G}\leq C_{6.7}e^{-(1-3\delta )G}$, this immediately gives a bound for $e^G$, hence $v$ at $p$. Then we are done. 

Suppose otherwise, then we have at $p$
\begin{equation}
0\geq e^{\delta G}\frac{1}{4}\Delta\phi-e^{\delta G}C_{6.5}\geq e^{\delta G}(\frac{1}{4}e^{\frac{G}{2}}-C_{6.5}).
\end{equation}
Then we also get an estimate for $e^G$ at $p$. So we are done as well.
\end{proof}

\noindent Xiuxiong Chen\\
University of Science and Technology of China and Stony Brook University\\

\noindent Jingrui Cheng\\
University of Wisconsin at Madison.

\end{document}